\documentclass[12pt]{article}
\usepackage{amsfonts} \input{epsf.sty} \usepackage[dvips]{graphicx}
\usepackage{latexsym,amsmath,amsfonts,amscd, amsthm} \usepackage{enumitem}
\usepackage{epsfig} \usepackage{changebar} \usepackage{pstricks}
\usepackage{pst-plot} \usepackage{multirow} \usepackage{placeins}
\usepackage{amssymb} \usepackage{stmaryrd} \usepackage{graphicx}
\usepackage{subcaption}
\usepackage[flushleft]{threeparttable} \topmargin-.5in \textheight9in
\newcommand{\ninti}{\overset{\sim}{\int_{I_i}}}
\newcommand{\nintI}{\overset{\sim}{\int_{I}}}
\newcommand{\nintj}{\overset{\sim}{\int_{J_j}}}
\newcommand{\nintJ}{\overset{\sim}{\int_{J}}}
\newcommand{\niintji}{\nintj\ninti}
\newcommand{\nintR}{\overset{\sim}{\int_\Omega}}

\newcommand{\intR}{\int_\Omega} 
\newcommand{\inti}{\int_{I_i}}
\newcommand{\intj}{\int_{I_j}}
\newcommand{\xx}{\tilde{x}} 
\newcommand{\yy}{\tilde{y}}
\newcommand{\mI}{\mathcal{I}} 
\newcommand{\sign}{\text{sign}}
\newcommand{\diag}{\text{diag}}
\newcommand{\vphi}{\varphi} 
\newcommand{\hf}{\frac{1}{2}}
\newcommand{\pre}{\textup{pre}}

\newtheorem{THM}{Theorem}[section] 
 
\newtheorem{LEM}{Lemma}[section] 
\newtheorem{REM}{Remark}[section] \theoremstyle{definition}
\newtheorem{examp}[subsubsection]{Example}
\newtheorem{Examp}[subsection]{Example}

\allowdisplaybreaks

\begin{document}

\baselineskip=2pc

\vspace*{.10in}

%%=============  title  =========================
\begin{center} 
  {\bf A discontinuous Galerkin method for nonlinear parabolic
  equations and gradient flow problems with interaction 
potentials} 
\end{center}

%\vspace{.1in}

\centerline{Zheng Sun\footnote{Division of Applied Mathematics, Brown
  University, Providence, RI 02912, USA. E-mail: zheng$\_$sun@brown.edu},
  Jos\'{e} A.
  Carrillo\footnote{Department of Mathematics, Imperial College London, London
SW7 2AZ, UK. E-mail: carrillo@imperial.ac.uk.
Research supported by the Royal Society via a Wolfson Research 
Merit Award and by EPSRC grant number EP/P031587/1.}  
  and 
  Chi-Wang Shu\footnote{Division of Applied Mathematics, Brown University,
    Providence, RI 02912, USA. E-mail: shu@dam.brown.edu.
Research supported by DOE grant DE-FG02-08ER25863 and 
NSF grant DMS-1418750.} }

\vspace{.2in}

\centerline{\bf Abstract} \bigskip

We consider a class of time dependent second order partial differential
equations governed by a decaying entropy. The solution usually
corresponds to a density distribution, hence positivity (non-negativity)
is expected.  This class of problems covers important cases such as Fokker-Planck
type equations and aggregation models, which have been studied intensively in
the past decades.  In this paper, we design a high order discontinuous Galerkin
method for such problems. If the interaction potential is not involved, or the
interaction is defined by a smooth kernel, our semi-discrete scheme admits an
entropy inequality on the discrete level.  Furthermore, by applying the
positivity-preserving limiter, our fully discretized scheme produces
non-negative solutions for all cases under a time step constraint.  Our method
also applies to two dimensional problems on Cartesian meshes.  Numerical
examples are given to confirm the high order accuracy for smooth test cases and
to demonstrate the effectiveness for preserving long time asymptotics.

\vfill


\clearpage

\section{Introduction} 
\label{sec-introduction} 
\setcounter{equation}{0}
\setcounter{figure}{0} 
\setcounter{table}{0}

%%***************************************************************************************
%                                Introduction
%****************************************************************************************

In this paper, we propose a discontinuous Galerkin method for solving the
initial value problem, 
\begin{equation} 
  \left\{ 
    \begin{aligned} &\partial_t \rho=
      \nabla\cdot\left(f(\rho)\nabla(H'(\rho)+V(\mathbf{x})+W\ast\rho)\right),
      \quad\mathbf x\in\Omega\subset \mathbb{R}^d,\quad t>0\label{eq-general},\\
      &\rho(\mathbf{x},0) = \rho_0(\mathbf{x}).\\ 
    \end{aligned}\right.  
\end{equation}
Here $\rho = \rho(\mathbf{x},t)\geq 0$ is a time-dependent density. 
$H'$ is strictly increasing and $W(\mathbf{x}) = W(-\mathbf{x})$ is
symmetric. We assume $f(\rho)\geq 0$ and it attains zero if 
and only if $\rho =0$. $f$ either increases strictly with
respect to $\rho$, or it satisfies the property $\frac{f(\rho)}{\rho}\leq C$
for some fixed constant $C$. 

Our concern for \eqref{eq-general} arises from two typical 
scenarios. The first
case is on nonlinear (possibly degenerate) parabolic equations,
\begin{equation}\label{eq-ndp} 
  \partial_t \rho = \nabla\cdot
  \left(f(\rho)\nabla\left(H'(\rho)+V(\mathbf{x})\right)\right).  
\end{equation} 
Many classic problems can be included in this setting, such as the heat 
equation, the porous media equation, the Fokker-Planck equation and so on.

In the second case, several authors take the interaction term $W\ast\rho$ into account,
while imposing $f(\rho) = \rho$. Hence the equation takes the form of a
continuity equation, and the velocity field is determined by the gradient of a
potential function.  
\begin{equation}\label{eq-Wgf} 
  \partial_t \rho = \nabla\cdot \left(\rho\mathbf{u}\right), \qquad
  \mathbf{u} = \nabla \xi =
  \nabla\left(H'(\rho)+V(x)+W\ast\rho\right).  
\end{equation} 
Here, $H$ is a density of internal energy, $V$ is a confinement potential, and
$W$ is an interaction potential. It is used to model, for example, interacting
gases \cite{carrillo2003kinetic,villani2003topics}, granular flows \cite{benedetto1997kinetic,BCCP} and
aggregation behaviors in biology \cite{mogilner1999non,topaz2006nonlocal}. This equation is
also related with the gradient flow for the Wasserstein metric on the space of
probability measures \cite{ambrosio2008gradient}. 

Both of the problems \eqref{eq-ndp} and \eqref{eq-Wgf} can be formulated under
the framework of \eqref{eq-general}. It has an underlying structure associated
with the entropy functional, 
\begin{equation}\label{eq-entropy} 
  E = \int_{\Omega}H(\rho(\mathbf{x})) d\mathbf{x}+ \int_{\Omega} V(\mathbf{x})
  \rho(\mathbf{x}) d\mathbf{x} +\frac{1}{2}\int_{\Omega}\int_{\Omega} 
  W(\mathbf{x}-\mathbf{y})\rho(\mathbf{x})\rho(\mathbf{y})
  d\mathbf{x}d\mathbf{y}.  
\end{equation} 
One can show that at least for classical solutions, 
\begin{equation}\label{eq-e-ed} 
  \frac{d}{dt}E(\rho) = -\int_{\Omega}f(\rho)|\mathbf{u}|^2 d\mathbf{x}
  : = - I(\rho) \le 0.  
\end{equation}
Here $\mathbf{u}$ is defined as that in \eqref{eq-Wgf} and $I$ is 
referred to as the entropy dissipation.

Indeed, \eqref{eq-e-ed} has provided much insight into the problem and has
helped people to study the dynamics of \eqref{eq-ndp} and \eqref{eq-Wgf}, see,
for example \cite{carrillo2001entropy,carrillo2003kinetic,villani2003topics}. Hence it
is desirable to develop numerical schemes mimicking a similar entropy-entropy
dissipation relationship in the discrete sense.

Another challenge for developing the numerical schemes is to ensure the
non-negativity of the numerical density without violating the mass conservation.
It is not only for the preservation of physical meanings, but also for the
well-posedness of the initial value problem. For example, in \eqref{eq-Wgf}, the
entropy may not necessarily decay if $\rho$ admits negative values. 

Numerical schemes addressing both of these concerns have been studied intensively
very recently. In \cite{bessemoulin2012finite}, the authors designed a second
order finite volume scheme for \eqref{eq-ndp}. Later on, a direct
discontinuous Galerkin method has been proposed by Liu and Wang in
\cite{liu2016entropy}.  Their scheme achieves high order accuracy but the
preservation of non-negativity only holds for limited cases. As for
\eqref{eq-Wgf}, a variety of numerical methods have been developed, including a
mixed finite element method \cite{burger2010mixed}, a finite volume method
\cite{carrillo2015finite}, a particle method \cite{carrillo2017numerical}, a
method of evolving diffeomorphisms \cite{carrillo2016numerical} and a blob
method \cite{craig2016blob} (for $H=0$, $V=0$).

In this paper, we design a high order discontinuous Galerkin (DG) method for
\eqref{eq-general}, which covers both \eqref{eq-ndp} and \eqref{eq-Wgf}. The DG
method is a class of finite element methods using spaces of discontinuous
piecewise polynomials and is especially suitable for solving
hyperbolic conservation
laws. Coupled with strong stability preserving Runge-Kutta (SSP-RK) time
discretization and suitable limiters (the so-called RKDG method, developed by 
Cockburn et al. in \cite{cockburn1991runge1, cockburn1989tvb2, cockburn1989tvb3, 
cockburn1990runge4, cockburn1998runge5}), the method captures shocks effectively 
and achieves high order accuracy in smooth regions \cite{cockburn2001runge}. 
The method has also been generalized for problems involving diffusion and higher
order derivatives, for example, the local DG method \cite{BR, cockburn1998local},
the ultra-weak DG method \cite{cheng2008discontinuous} and the direct DG method 
\cite{liu2009direct}. 

Our idea is to formally treat \eqref{eq-general} as a classical 
conservation law
and apply the techniques there to overcome the challenges. The main ingredients
for our schemes are:
\begin{enumerate}[topsep=0pt,itemsep=-1ex,partopsep=0ex,parsep=0ex] 
  \item Legendre-Gauss-Lobatto quadrature rule for numerical integration, 
  \item positivity-preserving limiter with SSP-RK time discretization.
\end{enumerate}
The quadrature is used to stabilize the semi-discrete scheme. Such approach has
already been studied in different contexts, such as in spectral collocation
methods \cite{hesthaven2007spectral} and in nodal discontinuous Galerkin methods
\cite{hesthaven2007nodal}. Recently, based on the methodology established in
\cite{carpenter2014entropy, gassner2013skew, gassner2016well}, Chen and Shu
proposed a unified framework for designing entropy stable DG scheme for
hyperbolic conservation laws using suitable quadrature rules
\cite{chen2017entropy}.  Their approach is also related with the
summation-by-part technique in finite difference methods. The
positivity-preserving limiter with provable high order accuracy is firstly
designed by Zhang and Shu in \cite{zhang2010maximum} to numerically ensure the
maximum principle of scalar hyperbolic conservation laws. Then the
methodology has been generalized for developing the bound-preserving schemes for
various systems. We refer to \cite{zhang2011maximum} and the references therein
for more details.  Our approach of implementing the positivity-preserving limiter for
parabolic problems is mainly inspired by the recent work of Zhang on
compressible Navier-Stokes equation \cite{zhang2017navier}, in which the author
considers the conservative form of the problem and introduces the diffusion 
flux to handle the second order derivatives.

Based on these techniques, we propose a discontinuous Galerkin scheme for
\eqref{eq-general} such that
\begin{enumerate}[topsep=0pt,itemsep=-1ex,partopsep=0ex,parsep=0ex]
  \item the semi-discrete scheme satisfies an entropy inequality for smooth $W$,
  \item the fully discretized scheme is positivity-preserving.
\end{enumerate} 
Our method also has other desired properties. It achieves high order accuracy,
conserves the total mass and preserves numerical steady states. Special care is needed for
the case of non-smooth interaction potential, due
to the fact that one should adopt exact integration to calculate 
the convolution, see \cite{carrillo2015finite}, as the Gauss-Lobatto quadrature is no longer accurate.

The remaining part of the paper is organized as follows. In Section
\ref{sec-1d-scheme}, we design the numerical method for one dimensional
problems. We firstly introduce the notations and briefly discuss our motivation
on deriving the scheme. Then it follows with the semi-discrete scheme and the
discrete entropy inequality. The next part is on the time discretization and the
positivity-preserving property of the fully discretized scheme. Finally we
outline the matrix formulation and the algorithm flowchart. Section
\ref{sec-2d-scheme} is organized similarly for two dimensional problems on
Cartesian meshes, while the implementation details are omitted. Then in Section
\ref{sec-num-result-1d} and Section \ref{sec-num-result-2d}, we present
numerical examples for one dimensional and two dimensional problems
respectively.  Finally conclusions are drawn in Section
\ref{sec-concluding-remarks}.

\section{Numerical method: one dimensional case}\label{sec-1d-scheme}
\setcounter{equation}{0} 
\setcounter{figure}{0}
\setcounter{table}{0}
%%***************************************************************************************
%                               Numerical method: 1D
%***************************************************************************************

\subsection{Notations and motivations}
\label{sec-prelim}

For now, we focus on the one dimensional case of \eqref{eq-general}
\begin{equation*} 
  \left\{ 
    \begin{aligned} 
      &\partial_t \rho = \partial_x\big(f(\rho)\partial_x(H'(\rho)+V(x)
      +W\ast\rho)\big),\quad x\in\Omega\subset \mathbb{R},\quad t>0
      \label{eq-general-1D},\\ 
      &\rho(x,0)= \rho_0(x).\\ 
    \end{aligned}
  \right.  
\end{equation*}
In general, the problem can be either defined on a connected compact domain
with proper boundary conditions, or it can involve the whole real line with
solutions vanishing at the infinity. In our numerical scheme, we will always
choose $\Omega$ to be a connected interval. For simplicity, 
the periodic or compactly supported boundary conditions are applied. But we
remark that our approach can be extended to more general types of boundary 
conditions, for example, with zero-flux boundaries.

Let $I_i = (x_{i-\hf},x_{i+\hf})$ and $I = \cup_{i=1}^N I_i$ be a regular
partition of the domain $\Omega$. Denote $h_i = x_{i+\hf}-x_{i-\hf}$ and $h =
\max_i h_i$.  We will seek a numerical solution in the discontinuous piecewise
polynomial space, \[V_h = \{v_h:v_h\big|_{I_i}\in P^k(I_i), \text{for }x\in
I_i,i=1,\dots,N\}.\] Here $P^k(I_i)$ is the space of $k$-th order polynomials
on $I_i$. Note that the functions in $V_h$ can be double-valued at the cell
interfaces. Hence the notations $v_h^+$ and $v_h^-$ are introduced for the
right limit and the left limit of $v_h$. For a function $s = s(\rho)$ or $s =
s(\rho,x)$, we denote by $s_h = s(\rho_h)$ or $s_h = s(\rho_h,x)$ respectively.
Furthermore, the notation $s_h^+$ stands for $s(\rho_h^+)$ or $s(\rho_h^+,x)$,
and $s^-_h$ stands for $s(\rho_h^-)$ or $s(\rho_h^-,x)$. 

To define the DG method, we split the original problem into the following system, 
\begin{subequations}
  \begin{align*} 
    \partial_t \rho &= \partial_x(f(\rho)u),\\ 
                  u &= \partial_x\xi,\\ 
                \xi &=  H'(\rho) + V + W\ast\rho.
  \end{align*}
\end{subequations} 
By applying the DG approximation, we obtain the following scheme as a
\emph{rudiment}.  We seek $\rho_h,u_h \in V_h$, such that for any test function
$\vphi_h, \psi_h \in V_h$,
\begin{subequations}\label{eq-1d-rudiment} 
  \begin{align} 
    \inti(\partial_t\rho_h) \vphi_h dx &= -\inti f_h u_h \partial_x \vphi_h dx +
    \widehat{f u}_{i+\hf} (\vphi_h)_{i+\hf}^- -\widehat{fu}_{i-\hf}(\vphi_h)_{i-\hf}^+,
      \label{sch-1dexac-1}\\ 
      \inti u_h \psi_h dx &= -\inti \xi_h \partial_x\psi_h dx + \widehat{\xi}_{i+\hf}
      (\psi_h)_{i+\hf}^- -\widehat{\xi}_{i-\hf}(\psi_h)_{i-\hf}^+. \label{sch-1dexac-2} 
  \end{align}
\end{subequations} 
Here $\xi_h = H'_h + V + \intR W(x-y)\rho_h(y) dx dy$, while $\widehat{f u}$ and
$\widehat\xi$ are numerical fluxes.

By setting $\vphi_h(x) = 1_{I_i}(x)$, one will get the following evolution
equation for the cell average of $(\rho_h)_i$, which is denoted by
$(\bar{\rho}_h)_i$.  
\begin{equation}\label{eq-1D-naive-average}
  \frac{d}{dt}\left(\bar{\rho}_h\right)_i =\frac{\widehat{f u}_{i+\hf}
-\widehat{f u}_{i-\hf}}{h_i}.  
\end{equation} 
This is very similar to that in hyperbolic conservation laws. In particular,
for $k=0$ and $u\equiv 1$, by using the so called monotone flux,
\eqref{eq-1D-naive-average} will become a monotone scheme, which satisfies many
desirable properties.

In order to achieve an entropy-entropy dissipation relationship as that for the
exact solution, one may hope to set $\vphi_h = \xi_h$ in \eqref{sch-1dexac-1}
and $\psi_h = u_hf_h$ in \eqref{sch-1dexac-2}. Unluckily, neither of them
falls into the test function space. A natural attempt is to use the usual $L^2$
projection to enforce this property, as that in \cite{liu2016entropy}. However,
the projection will change the values at the cell interfaces, and the desired
form in \eqref{eq-1D-naive-average} will be violated. This will cause trouble
when one seeks to preserve the non-negativity of the solution. Inspired by the
recent work of Chen and Shu in \cite{chen2017entropy}, we introduce a
suitable quadrature to overcome this difficulty. Moreover, the
Gauss-Lobatto quadrature is
used to preserve the values at the cell ends.

Let us denote by $\{x^r_i\}_{r=1}^{k+1}$ the $k+1$ Gauss-Lobatto quadrature
points on $I_i$ and $\{w_r\}_{r=1}^{k+1}$ the $k+1$ Gauss-Lobatto quadrature
weights on $[-1,1]$. In particular $x^1_i = x^+_{i-\hf}$ and $x^{k+1}_i =
x^-_{i+\hf}$. On each cell $I_i$, the operator $\mI$ returns the $k$-th order 
polynomial interpolating at $\{x^r_i\}_{r=1}^{k+1}$.  We will use the
notation 
\[\ninti \eta\zeta dx=
\frac{h_i}{2}\sum_{r=1}^{k+1} w_r \eta(x_i^r)\zeta(x_i^r)\] 
and 
\[\ninti\eta\partial_x\zeta dx= \frac{h_i}{2}\sum_{r=1}^{k+1} w_r
\eta(x_i^r)(\partial_x\mI\zeta)(x_i^r)\]
for the Gauss-Lobatto quadrature. As a convention, $\nintR$ stands for $\sum_i\ninti$.

We will finally come to the fully discretized scheme, for which
we denote by $\tau$ the 
time step length and $\lambda_i = \frac{\tau}{h_i}$.

\subsection{Semi-discrete scheme and entropy inequality}\label{sec-1d-semi}

Our scheme is to replace the integrals in \eqref{eq-1d-rudiment} by
the Gauss-Lobatto quadrature. In other words, we seek $\rho_h, u_h \in{V_h}$, such
that for any test functions $\vphi_h, \psi_h \in V_h$,
\begin{subequations}\label{sch-1D} 
  \begin{align} 
    \ninti (\partial_t\rho_h) \vphi_h dx &= -\ninti f_h  u_h \partial_x{\vphi_h} dx 
    + \widehat{fu}_{i+\hf} (\vphi_h)_{i+\hf}^- -\widehat{fu}_{i-\hf}(\vphi_h)_{i-\hf}^+,
    \label{eq-sch-1dnum-1}\\ \ninti u_h \psi_h dx
    &= -\ninti \xi_h \partial_x\psi_h dx + \widehat{\xi}_{i+\hf}(\psi_h)_{i+\hf}^- -
    \widehat{\xi}_{i-\hf}(\psi_h)_{i-\hf}^+. \label{eq-sch-1dnum-2} 
  \end{align}
\end{subequations} 
Here, when the interaction potential $W$ is smooth, we set
\begin{equation*} 
  \xi_h = \xi_h(\rho_h,x) = H'_h + V + \nintR W(x-y)\rho_h(y)dy.  
\end{equation*} 
While for non-smooth $W$, the quadrature may not achieve sufficient accuracy of the 
convolution. Hence the exact integration is applied  
\begin{equation*} 
  \xi_h = \xi_h(\rho_h,x) = H'_h + V + \intR W(x-y)\rho_h(y) dy.  
\end{equation*} 
The numerical fluxes are chosen in the following way, 
\begin{subequations} 
  \begin{align*}  
    \widehat{\xi}&= \frac{1}{2}(\xi_h^+ + \xi_h^-),\\ 
    \widehat{fu} &= \frac{1}{2}(f_h^+u_h^++ f_h^-u_h^-)+\frac{\alpha}{2}(g_h^+-g_h^-),
    \quad \alpha = \max\{|u^+_h|,|u^-_h|\}, 
  \end{align*} 
\end{subequations} 
with $g(\rho) = f(\rho)$ if $f$ is increasing and $g(\rho) = C\rho$ if
$\frac{f(\rho)}{\rho}\leq C$. $\widehat{\xi}$ is the central flux. When $f(\rho)
= \rho$, one can choose $g(\rho) = \rho$ and $\widehat{fu}$ is the 
Lax-Friedrich flux. 

\begin{REM}
  Although we formally require that $\varphi_h$ and $\psi_h$ are taken from the
  finite element space, our scheme \eqref{sch-1D} actually can not distinguish a
  function from its interpolation polynomial at the Gauss-Lobatto points.
  Hence, \eqref{sch-1D} will also hold for ``test functions'' outside $V_h$.
  This facilitates our proof of the discrete entropy inequality. This fact can
  also be justified through the matrix formulation, which is presented in Section
  \ref{sec-imp}.
\end{REM}

This semi-discrete scheme satisfies the following entropy inequality.  
\begin{THM} \label{thm:main}
For smooth interaction kernel $W$, assume that the semi-discrete scheme \eqref{sch-1D} has a solution, then it satisfies a similar entropy-entropy dissipation relationship, as that in \eqref{eq-e-ed}, given by
  \begin{equation*}%\label{eq-disc-ei} 
    \frac{d}{dt}\tilde{E} \leq -\tilde{I}, 
  \end{equation*} 
  where 
  \begin{equation*}
    \tilde{E} = \nintR H(\rho_h) dx + \nintR V\rho_h dx + \frac{1}{2}
    \nintR\nintR W(x-y)\rho_h(x)\rho_h(y) dx dy
  \end{equation*} is the discrete entropy and 
  \begin{equation*} 
    \tilde{I} =  \nintR f_h|u_h|^2 dx
  \end{equation*} is the associated discrete entropy dissipation. 

  {Indeed, one can choose larger $\alpha_{i+\hf}$ in the 
Lax-Friedrich flux. This will bring in extra numerical dissipation 
and the entropy inequality will still hold. Moreover, if 
$\alpha_{i+\hf}>0$ for all $i$ and the semi-discrete scheme 
\eqref{sch-1D} achieves a non-negative stationary 
state $\rho_h$, then $\rho_h$ is continuous and the piecewise 
polynomial interpolation of $\xi_h$ is constant in each 
connected component $J$ of the support of $\rho_h$, which is 
defined by $J=\cup_{i\in\Lambda} I_i$ for certain set of 
consecutive indices $\Lambda$ where $\rho_h>0$.}
\end{THM}
\begin{proof} 
Using the symmetry of $W$, we have 
\begin{equation*} 
  \begin{split}
    &\frac{d}{dt}\frac{1}{2} \nintR\nintR W(x-y)\rho_h(x)\rho_h(y) dx dy\\ 
    =& \frac{1}{2}\nintR \partial_t\rho_h(x)\left( \nintR W(x-y) \rho_h(y)dy\right)dx 
    + \frac{1}{2}\nintR \partial_t\rho_h(y)\left( \nintR W(x-y) \rho_h(x) dx\right)dy\\ 
    =& \nintR \partial_t \rho_h(x)\left( \nintR W(x-y)\rho_h(y) dy\right) dx. 
  \end{split} 
\end{equation*} 
Hence,
\begin{equation*} 
  \begin{split} 
    \frac{d}{dt}\tilde{E} &=
      \nintR\partial_t\rho_h(x)\left(H'_h + V +\nintR W(x-y)\rho_h(y)dy\right) dx\\
      &=  \nintR\partial_t\rho_h \xi_h dx = \sum_i\ninti \partial_t\rho_h \mI(\xi_h) dx\\ 
      &= \sum_i \left(-\ninti\mI(f_hu_h)\partial_x\mI(\xi_h)dx + 
      \widehat{fu}_{i+\hf}(\xi_h)_{i+\hf}^--\widehat{fu}_{i-\hf}(\xi_h)_{i-\hf}^+\right).
  \end{split}
\end{equation*} 
Note $\mI(f(\rho_h)u_h)\partial_x\mI(\xi_h)$ is a polynomial of degree $2k-1$
and the Gauss-Lobatto quadrature with $k+1$ points is exact. Hence we can
replace the quadrature with the exact integral and integrate by parts. Then we
obtain
\begin{equation*} 
  \begin{split} 
    \frac{d}{dt}\tilde{E}= \sum_i &\left(-\inti\mI(f_hu_h)\partial_x\mI(\xi_h)dx 
      + \widehat{fu}_{i+\hf}(\xi_h)_{i+\hf}^-
      -\widehat{fu}_{i-\hf}(\xi_h)_{i-\hf}^+\right)\\
       = \sum_i &\bigg(\inti (\partial_x \mI(f_hu_h))\mI(\xi_h)dx
      -(f_hu_h\xi_h)_{i+\hf}^-\\
      \qquad\qquad&+(f_h u_h\xi_h)_{i-\hf}^+ + \widehat{f
        u}_{i+\hf}(\xi_h)_{i+\hf}^-
      -\widehat{fu}_{i-\hf}(\xi_h)_{i-\hf}^+\bigg).\\
  \end{split} 
\end{equation*} 
By using the same trick, we change the exact integral back to the Gauss-Lobatto
quadrature, and apply the scheme \eqref{eq-sch-1dnum-2} to obtain 
\begin{align}\label{entrdis}  
    \frac{d}{dt}\tilde{E} =  &\sum_i\left(\ninti \xi_h
      \partial_x (f_h u_h)dx - (f_h u_h\xi_h)_{i+\hf}^-+
      (f_hu_h\xi_h)_{i-\hf}^+ + \widehat{fu}_{i+\hf}(\xi_h)
    _{i+\hf}^--\widehat{fu}_{i-\hf}(\xi_h)_{i-\hf}^+\right)\nonumber\\ 
    =&\sum_i \bigg(-\ninti f_h|u_h|^2dx+\widehat{\xi}_{i+\hf}
      (f_hu_h)_{i+\hf}^--\widehat{\xi}_{i-\hf}(f_hu_h)_{i-\hf}^+
    - (f_hu_h\xi_h)_{i+\hf}^-\nonumber\\
    &+(f_hu_h\xi_h)_{i-\hf}^+ +\widehat{f u}_{i+\hf}(\xi_h)_{i+\hf}^-
    -\widehat{fu}_{i-\hf}(\xi_h)_{i-\hf}^+\bigg)\nonumber\\ 
    = &-\nintR f_h |u_h|^2 dx - \sum_i
    \frac{\alpha_{i+\hf}}{2}[g_h]_{i+\hf}[\xi_h]_{i+\hf},  
\end{align}
where the bracket represents the jump, $[g_h] = g_h^+-g_h^-$. According to our
choice of $g$, $\sign[g_h] = \sign[\rho_h]$. 
{
  Since $V$ and $W$ are single-valued 
  functions, $[\xi_h] = [H'(\rho_h)]$. By using the fact that 
  $H'$ is strictly increasing, 
  we have $\sign[\xi_h] = \sign[H'(\rho_h)] =\sign[\rho_h]$. 
  Therefore $\sum_i \frac{\alpha_{i+\hf}}{2}[g]_{i+\hf}
[\xi_h]_{i+\hf} \geq 0$, which completes the proof of the first claim.}

{
  It is easy to see that the entropy inequality will hold as long as the coefficients $\alpha_{i+\hf}$ are non-negative. Let us assume now that $\alpha_{i+\hf}>0$ for all $i$ and $\rho_h$ is a non-negative 
stationary state of the semi-discrete scheme \eqref{sch-1D}, namely 
\[\frac{d}{dt}\tilde{E} = 0.\]
Then from 
\eqref{entrdis} we deduce that both terms in the right-hand side must vanish, that is
$$
\sum_i \ninti f_h |u_h|^2 dx = \sum_i \frac{\alpha_{i+\hf}}{2}[g_h]_{i+\hf}[\xi_h]_{i+\hf}=0\,.
$$
Therefore each term in the summations will be zero. On the one hand, if $\alpha_{i+\hf}>0$, then $[g_h]_{i+\hf}[\xi_h]_{i+\hf} = 0$ and hence $[\rho_h]_{i+\hf} = 0$. It holds for all $i$, which means the jumps of $\rho_h$ vanish on all the cell interfaces. This implies the continuity of $\rho_h$.
On the other hand,} in the interval $J$ we deduce that $u_h=0$ due to the positivity of $f_h$ implied 
by the positivity of $\rho_h$ and its definition. Hence for all 
$i\in\Lambda$, by using \eqref{eq-sch-1dnum-2}, one can obtain
\begin{equation*}
  \begin{split}
    \inti \partial_x(\mI\xi_h) \mI\psi_h dx &= -\inti \mI\xi_h \partial_x(\mI \psi_h)
    dx + (\xi_h\psi_h)_{i+\hf}^- - (\xi_h\psi_h)_{i-\hf}^+ \\
    &=  -\ninti \xi_h \partial_x\psi_h dx 
    + \widehat{\xi}_{j+\hf}(\psi_h)_{i+\hf}^- - \widehat{\xi}_{i+\hf}(\psi_h)_{i-\hf}^+ \\
    &= 0.
  \end{split}
\end{equation*}
Here $\xi_h^\pm = \widehat{\xi}$ is guaranteed by the continuity of $\xi_h$ (implied by 
that of $\rho_h$). Therefore, $\mI \xi_h$ is constant on each $i \in \Lambda$. 
Due to the fact that $\xi_h$ is continuous globally, all these constants must be the 
same and the piecewise polynomial interpolation of $\xi_h$ is constant on $J$.
\end{proof}

\subsection{Time discretization and preservation of positivity}

The semi-discrete scheme itself does not guarantee the positivity of the
numerical solution. If no special treatment is applied, one may produce
nonsense density with negative values and the problem can become illposed.
Hence we adopt the methodology developed by Zhang and Shu in
\cite{zhang2010maximum}, which enforces the positivity of the solution without
violating the mass conservation. Their idea is to incorporate a
positivity-preserving limiter into the strong stability preserving Runge-Kutta
(SSP-RK) time discretization. Under certain time step constraints, each Euler
forward step preserves the positivity of the cell average (referred to
as the weak
positivity in the literature). Then one can scale the solution, without
affecting spatial accuracy, to ensure the point-wise non-negativity. 
SSP-RK
time discretization will preserve the non-negativity of the solution in 
the Euler forward steps.

\subsubsection{First order Euler forward in time}

Let us firstly consider the Euler forward time stepping. We use the superscript
``$\pre$'' for the solution obtained by Euler forward method before applying the
positivity-preserving limiter.  The time discretization of
\eqref{eq-sch-1dnum-1} becomes 
\begin{equation}\label{sch-1dnum-full-1} 
  \ninti \frac{\rho_h^{n+1,\pre}-\rho_h^n}{\tau} \vphi_h dx = -\ninti f_h^nu_h^n
  \partial_x{\vphi_h} dx + (\widehat{f u})^n_{i+\hf} (\vphi_h)_{i+\hf}^-
  -(\widehat{f u})^n_{i-\hf}(\vphi_h)_{i-\hf}^+.  
\end{equation}
\begin{LEM}\label{lem-1D-posi-aver} 
  Suppose $\rho_h^n(x_i^r)\geq 0$ at the Gauss-Lobatto quadrature points.  Then
  when $\lambda_i \leq \min\{(\frac{w_1\rho}{fu+\alpha g})_{i-\hf}^+,
  (\frac{w_{k+1}\rho}{\alpha g-fu})_{i+\hf}^-\}$, the solution $\rho_h^
  {n+1,\pre}$ obtained from \eqref{sch-1dnum-full-1} satisfies $\left(
  \bar\rho_h\right)_i^{n+1,\pre}\geq 0$. Here, in the constraint of
  $\lambda_i$, we consider $\frac{0}{0} := +\infty$ by convention.
\end{LEM}
\begin{proof} 
  We drop all subscripts $h$ in this proof. Take $\vphi = 1$ in \eqref{sch-1dnum-full-1}, we have 
  \begin{equation*} 
    \bar{\rho}^{n+1,\pre}_i = \bar{\rho}^n_i+\lambda_i\left({(\widehat{fu})^n_{i+\hf}
    -(\widehat{fu})^n_{i-\hf}}\right).
  \end{equation*} 
  Note that $\rho^n$ is a polynomial of degree $k$, the Gauss-Lobatto quadrature
  is exact for evaluating the cell average $\bar{\rho}_i^n$. More specifically, we
  have 
  \[\bar {\rho}_i^n =\frac{1}{h_i}\int_{I_i} \rho^n dx = 
  \sum_{r = 1}^{k+1}\frac{w_r}{2}\rho^n(x_i^r).\]
  The superscripts $n$ will also be
  omitted for simplicity in the rest. Hence 
  \begin{equation*} 
    \begin{aligned} 
      \bar{\rho}_i^{n+1,\pre} =&
      \sum_{r = 2}^k \frac{w_r}{2} \rho(x^r_i) + \frac{w_1}{2}\rho_{i-\hf}^+ +
      \frac{w_{k+1}}{2}\rho_{i+\hf}^- + \frac{\lambda_i}{2}\left( (fu)^+_{i+\hf}+
      (fu)^-_{i+\hf} + \alpha_{i+\hf}
    \left(g_{i+\hf}^+-g_{i+\hf}^-\right)\right) \\ 
    &-\frac{\lambda_i}{2}\left((fu)^+_{i-\hf}+ (fu)^-_{i-\hf} + \alpha_{i-\hf}
    \left(g_{i-\hf}^+-g_{i-\hf}^-\right)\right) \\ 
    =& \sum_{r = 2}^k \frac{w_r}{2}\rho(x^r_i) + 
      \left(\frac{w_{k+1}}{2}\rho^-_{i+\hf}+\frac{\lambda_i}{2}\left( 
          (fu)^-_{i+\hf} -\alpha_{i+\hf}g_{i+\hf}^-\right)\right) \\
          &+ \left(\frac{\omega_1}{2}\rho_{i-\frac{1}{2}}^+ 
          -\frac{\lambda_i}{2}\left((fu)^+_{i-\hf}+\alpha_{i-\hf}g^+_{i-\hf}
          \right)\right) \\
    &+ \frac{\lambda_i}{2}\left((fu)^+_{i+\hf}+\alpha_{i+\hf}g_{i+\hf}^+\right) 
    - \frac{\lambda_i}{2}\left((fu)_{i-\hf}^- - \alpha_{i-\hf}g_{i-\hf}^-\right).\\ 
  \end{aligned}
\end{equation*} 
  The first term is automatically non-negative, since the weights $w_r\geq 0$ and
  the nodal values $\rho(x^r_i)\geq 0$. The positivity of the last two terms is
  guaranteed by our choice of $\alpha$ and $g$. One only needs $\lambda_i\leq\min
  \{(\frac{w_1\rho}{fu+\alpha g})_{i-\hf}^+,(\frac{w_{k+1}\rho}{\alpha
  g-fu})_{i+\hf}^-\}$ to ensure the second and the third term to be non-negative.
  (Note that for $\rho_{i-\hf}^+$ or $\rho_{i+\hf}^-$ being $0$, the corresponding
  term is also $0$ and there is nothing to impose. Hence we introduce the
  notation $\frac{0}{0}:=+\infty$.) Therefore $\bar{\rho}^{n+1,\pre}_i\geq 0$ under
  the prescribed time step constraint.
\end{proof}
\begin{REM} 
  \begin{enumerate}
  \
    \item According to the definition of $\alpha$ and $g$, 
      $(\frac{w_1\rho}{fu+\alpha g})_{i-\hf}^+$ and 
      $(\frac{w_{k+1}\rho}{\alpha g-fu})_{i+\hf}^-$ will always be
      non-negative. 

    \item Although the original equation can be parabolic, we have incorporated
      the second order derivative into $u$, such that one can formally treat it
      as a hyperbolic problem. This technique is introduced by Zhang for the 
      compressible Navier-Stokes equation \cite{zhang2017navier}. In particular, for
      $f(\rho) = \rho$, one obtains $\lambda_i(\alpha\pm u^\pm)\leq C$, If $u$
      is a constant (as that in Example \ref{examp:adv}), this is the usual CFL
      condition. While in general, the bound of $\lambda_i$ may scale like
      $h_i$ and it then gives a typical time step restriction for parabolic
      problems.  

  \end{enumerate}
\end{REM}

Lemma \ref{lem-1D-posi-aver} tells us an inherent property of the Euler forward
scheme. If the solution is non-negative at the previous time step (at the
Gauss-Lobatto quadrature points), as long as the time step is smaller than a
threshold, the cell average at next time step
will remain non-negative.  In order to close the
loop, one would need to ensure the nodal values at the quadrature points 
of the next time step are also
non-negative. This indeed can be achieved by applying a scaling limiter, which
luckily does not affect the spatial accuracy. We refer to 
\cite{zhang2017navier} for more details.

\begin{LEM}\label{lem-EF-scaling}
  Let 
  \[\rho_h^{n+1}(x_i^r) = (\bar \rho_h)_i^{n+1,\pre}+\theta_i\left(\rho_h^{n+1,\pre}
  (x_i^r) - (\bar\rho_h)_i^{n+1,\pre}\right),\qquad\forall r = 1,\dots,k+1,\] 
  with $\theta_i=\min\{\frac{ (\bar \rho_h)_i^{n+1,\pre} }{(\bar \rho_h)_i^{n+1,\pre} -
  m_i},1\}$ and $m_i = \min\{\rho_h^{n+1,\pre}(x_i^r)\}_{r=1}^{k+1}$. Then we
  have $\rho_h^{n+1}(x_i^r)\geq0,\forall r = 1,\dots,k+1$ and $\left(\bar{\rho}_h\right)
  _i^{n+1} =\left(\bar{\rho}_h\right)_i^{n+1,\pre}$. Furthermore, the interpolation 
  polynomial of $\{\rho_h^{n+1}(x_i^r)\}$ on $I_i$ satisfies \[|\rho_{h}^{n+1}(x) -
  \rho_h^{n+1,\pre}(x)|\leq C_k\max_{x\in I_i}|\rho(x,t_{n+1})-\rho_h^{n+1,\pre}(x)|
  ,\] where $\rho(x,t_{n+1})$ is the exact solution at time $t_{n+1}$ and $C_k$ is a 
  constant depending only on the polynomial degree $k$.
\end{LEM}

\begin{REM} 
    Our scheme only uses the nodal values at the Gauss-Lobatto quadrature
      points, hence we only need to ensure the non-negativity at these nodes.
      One can also squash the solution polynomials so that the solution is
      non-negative everywhere on the domain.  The proof will still go through.
\end{REM}

\begin{THM}\label{thm-1d-euler}
  With the scaling limiter in Lemma \ref{lem-EF-scaling}, the Euler forward time
  discretization of the semi-discrete scheme is positivity-preserving, provided
  the time step restriction specified in Lemma \ref{lem-1D-posi-aver} is satisfied.  
\end{THM}

\subsubsection{High order time discretization}

The SSP-RK method will be used for time discretization. We refer 
readers to
\cite{gottlieb2001strong} for more details. Since the time 
step scales like $\tau =
Ch^2$, the Euler forward method will be sufficient for piecewise linear elements
to achieve overall second order accuracy. For $k = 2,3$, we will use the second
order SSP-RK scheme  
\begin{subequations}\label{eq-SSP-RK-2} 
  \begin{align}
  \rho_h^{(1)} &= \rho^n_h + \tau F(\rho_h^n),\\ 
  \rho_h^{n+1} &= \frac{1}{2}\rho_h^n + \frac{1}{2}\left(\rho_h^{(1)} 
+ \tau F(\rho_h^{(1)})\right).  
  \end{align}
\end{subequations} 
For $k = 4,5$, the third order SSP-RK scheme is used
\begin{subequations}\label{eq-SSP-RK-3} 
  \begin{align}
  \rho_h^{(1)} &= \rho^n_h + \tau F(\rho_h^n),\\ 
  \rho_h^{(2)} &= \frac{3}{4} \rho_h^n + \frac{1}{4}
  \left(\rho^{(1)}_h + \tau F(\rho_h^{(1)})\right),\\ 
  \rho_h^{n+1} &= \frac{1}{3}\rho_h^n + \frac{2}{3}\left(\rho_h^{(2)} 
    + \tau F(\rho_h^{(2)})\right).  
  \end{align}
\end{subequations}

The positivity-preserving limiter should be applied immediately after each
Euler forward stage. As one can see, the SSP-RK schemes \eqref{eq-SSP-RK-2} and
\eqref{eq-SSP-RK-3} can be rewritten as convex combinations of the Euler
forward steps. Since each Euler forward step preserves the positivity, the
numerical density at the next time level will remain non-negative. 

\begin{THM} 
  Consider the SSP-RK time discretization \eqref{eq-SSP-RK-2} and
  \eqref{eq-SSP-RK-3} of the semi-discrete scheme \eqref{sch-1D}.  By applying
  limiters specified in Lemma \ref{lem-EF-scaling}, the fully discretized
  scheme preserves non-negativity as long as the time step restriction in Lemma
  \ref{lem-1D-posi-aver} is satisfied.  
\end{THM}

We also mention several other properties of the fully discretized scheme, whose
proofs are omitted. Such properties also hold for two dimensional cases.

\begin{enumerate} 
  \item Mass conservation: $\intR \rho_h^n(x) dx = \intR\rho_0(x) dx$.  
{
  \item Preservation of numerical steady states: if the numerical potential $\xi_h$ 
    becomes constant on each connected component of the support of $\rho_h$, 
    and $\rho_h$ vanishes everywhere else, then we have $\rho_h^{n+1} = \rho_h^n$. 
    We remind the readers that the ``preservation of steady states'' 
    here for the fully discretized scheme is slightly different from that in Theorem 
    \ref{thm:main} for the semi-discrete scheme with smooth $W$. But they are 
    related through the profiles of $\rho_h$ and $\xi_h$.}
\end{enumerate}

\subsection{Matrix formulation and implementation}\label{sec-imp}

At the end of this section, we would like to introduce the matrix formulation of
our numerical scheme and outline the flowchart of the algorithm. 

\subsubsection{Matrix formulation}

The derivation of the matrix formulation is similar to that in Section 3.1 of
\cite{chen2017entropy}.  We refer to that paper for more details. 

We omit all the subscripts $h$. Let $\{\zeta_r\}_{r=1}^{k+1}$ be the
Gauss-Lobatto quadrature points on the reference element $[-1,1]$.  We denote by
$L_r$, $r = 1,\dots,k+1$ the Lagrangian basis polynomials interpolating at these
nodes. 
\[L_r(\zeta) = \prod^{k+1}_{s=1,s\neq r}
\frac{\zeta-\zeta_s}{\zeta_r-\zeta_s}.\]
On each cell, the unknown function can be represented as 
\[{\rho}(x,t) =\sum_{r=1}^{k+1}\rho_i^r(t)L_r(\zeta^i(x)),\qquad x\in I_i.\]
Here $\zeta^i$ is the mapping from $I_i$ to $[-1,1]$.  To determine $\rho$, it
suffices to identify the coefficients $\vec{\rho}_i =
[\rho_i^1 \dots \rho_i^{k+1}]^T$. $\vec{u}_i$ and $\vec{\xi}_i$ are defined in a
similar fashion. 

The matrix formulation can be written as follows.
\begin{subequations}
  \begin{align}
    \frac{d}{dt}\vec{\rho}_i &= -\frac{2}{h_i}M^{-1}D^TM\vec{f u}_i 
    +\frac{2}{h_i}M^{-1}B\vec{fu}_i^*,\label{eq-mat-rho}\\
    \vec{u}_i &= -\frac{2}{h_i}M^{-1}D^TM\vec{\xi}_i 
    +\frac{2}{h_i}M^{-1}B\vec{\xi}_i^*,\label{eq-mat-u}\\
    \xi_i^r &= H'(\rho_i^r) + V(x_i^r)+\sum_j\rho_j^r\intj 
    W(x_i^r-y) L_r(\zeta^j(y))dy.\label{eq-mat-xi}
  \end{align}
\end{subequations}
Here $M = \diag\{w_1,\dots,w_{k+1}\}$ and $B = \diag\{-1,0,\dots,0,1\}$.
$D=\left( D_{rs}\right)$ is the difference matrix, and $D_{rs}=L_s'(\zeta_r)$.
$\vec{fu}_i$ is the component-wise product of $\vec{f}_i=[f(\rho_i^1) \dots
f(\rho_i^{k+1})]^T$ and $\vec{u}_i$. $\vec{\xi}_i^* = [\widehat{\xi}_{i-\hf}\ 0\
\dots \ 0 \ \widehat{\xi}_{i+\hf}]^T$ and $\vec{fu}_i^*$ is defined similarly for
$\widehat{fu}$. We remind the readers that one should replace $\inti$ by $\ninti$
in \eqref{eq-mat-xi} if $W$ is smooth.

\subsection{Algorithm flowchart}

For simplicity, we only consider the Euler forward time stepping. The algorithm
with SSP-RK time discretization can be implemented by repeating the following
flowchart in each stage.

\begin{enumerate}
  \item\label{1} Use \eqref{eq-mat-xi} to obtain $\{(\vec{\xi}_i)^n\}$.
  \item Evaluate the numerical flux $\{(\vec{\xi}_i^*)^n\}$ and use
    \eqref{eq-mat-u} to update $\{(\vec{u}_i)^n\}$. 
  \item Evaluate $\{(\vec{f}_i)^n\}$, $\{(\vec{fu}_i)^n\}$ and the numerical flux
    $\{(\vec{fu}_i^*)^n\}$. Use Euler forward time stepping for \eqref{eq-mat-rho}
    to calculate $(\vec{\rho}_i)^{n+1,\pre}$.
  \item Evaluate $\{\bar{\rho}_i^n\}$ in each cell.
    \begin{itemize}
      \item If $\{\bar{\rho}_i\}$ is a set of non-negative numbers. Apply the
        positivity preserving limiter to obtain $\{(\vec{\rho}_i)^{n+1}\}$ and
        enter the next time level.
      \item Otherwise halve the time step $\tau$ and restart from \ref{1}.
    \end{itemize}
\end{enumerate}

\begin{REM}
  \begin{enumerate}
    \
    \item The main advantage for using Gauss-Lobatto interpolation polynomial
      basis is that all the needed nodal values are automatically acquired.
      Hence one can save costs on evaluating the numerical fluxes and applying
      the positivity-preserving limiter. 
    \item For $W\neq 0$, the computational bottleneck is to calculate the
      convolution in step \ref{1}. This usually takes 
$\mathcal{O}(N^2)$ operations in each
      iteration.  However, on uniform meshes, the fast Fourier transform (FFT)
      can be applied to reduce the cost to 
$\mathcal{O}(N\log N)$. The idea is that, for 
      each fixed $i$, the convolution can be evaluated by,
      \begin{equation}\label{eq-convolu}
        (\vec{\xi}_i)^n =\sum_{m} K_m (\vec{\rho}_{i+m})^n, 
        \qquad (K_m)_{rs} = \int_{I_1}W\left(x_i^r
        -\left((m-1)h+y\right)\right)L_s(\zeta^1(y))dy.
      \end{equation}
      If the convolution kernel is periodic, then \eqref{eq-convolu} can be
      formulated as the multiplication of an $\left(N\times(k+1)\right)
      \times\left(N\times(k+1)\right)$ block circulant matrix and a vector. The
      FFT acceleration for such problems is standard.
      
      Although $W$ is not periodic most of the time, $\rho$ is usually a
      (numerically) compactly supported function. One only needs to evaluate
      $\xi$ precisely on the same interval. Hence we can simply extend the
      problem to a larger domain to adopt the previous procedure. For example,
      if $\rho$ lives on $[-R,R]$. We can consider its zero extension on
      $[-2R,2R]$ and assume everything to be $4R$ periodic. When the FFT
      algorithm is used to compute the matrix multiplication, it gives exact
      $\xi$ on $[-R,R]$, because relevant values of $W$ is unchanged on
      $[-2R,2R]$.  The computational complexity is still
      $\mathcal{O}\left(N\log N\right)$.
    \item In our numerical tests, both for one dimensional and two 
      dimensional cases, we will use a sufficiently small time step 
      to avoid the cell average attaining negative values. 
Also, $g(\rho) = \rho$ 
      will be used to define the numerical flux, unless otherwise stated.
  \end{enumerate}
\end{REM}

%****************************************************************************************
\section{Numerical method: two dimensional case} 
\label{sec-2d-scheme}
\setcounter{equation}{0} 
\setcounter{table}{0}
%****************************************************************************************
%                     Numerical method: two dimensional case
%****************************************************************************************

In this section, we apply our method to solve two dimensional problems on 
Cartesian meshes. 

\subsection{Semi-discrete scheme and entropy inequality} 

Consider the initial value problem,
\begin{equation*} 
  \left\{ 
    \begin{aligned} 
      &\partial_t \rho =
      \nabla\cdot\left(f(\rho)\nabla(H'(\rho)+V(x,y)+W\ast\rho)\right),\quad
      x,y\in\Omega\subset\mathbb{R}^2,\quad t>0\label{eq-general-2D},\\ 
      &\rho(x,y,0) = \rho_0(x,y).\\ 
  \end{aligned}\right.  
\end{equation*} 
Here $\Omega=I\times J$ is a rectangular domain and the periodic boundary
conditions are applied. Let $I_i\times J_j =
[x_{i-\hf},x_{i+\hf}]\times[y_{j-\hf},y_{j+\hf}]$ and
$\cup_{i=1}^{N_x}\cup_{j=1}^{N_y} I_i\times J_j$ be a partition of the mesh.
The mesh size is denoted by $h = \max_{i,j}\{\sqrt{(h^x_i)^2+(h^y_j)^2}\}$,
where $h^x_i=x_{i+\hf}-x_{i-\hf}$ and $h^y_j=y_{j+\hf}-y_{j-\hf}$. The finite
element spaces are defined as \[V_h = \{v_h:v_h\big|_{I_i\times J_j} \in
Q^k(I_i\times J_j), \text{ for all }i=1,\dots,N_x,j=1,\dots,N_y\} \text{ and }
\mathbf{V}_h = V_h\times V_h.\] Here $Q^k(I_i\times J_j)$ is the tensor
product space of $P^k(I_i)$ and $P^k(J_j)$.

The semi-discrete DG scheme is formulated as follows. One needs to find $\rho_h
\in V_h$ and $\mathbf{u}_h=(u^x_h,u^y_h)\in \mathbf{V}_h$, such that 
\begin{equation}\label{sch-2dnum-1} 
  \begin{split} 
    \niintji (\partial_t\rho_h) \vphi_h dxdy =& -\niintji f_h \left(  u^x_h
    \partial_x{\vphi_h}+u^y_h\partial_y{\vphi_h}\right) dxdy \\ 
    &+ \nintj\widehat{f u^x}_{i+\hf}\varphi_h(x_{i+\hf}^-,y)-\widehat{f
    u^x}_{i-\hf}\varphi_h(x_{i-\hf}^+,y)dy \\ 
    &+\ninti \widehat{f u^y}_{j+\hf} \vphi_h(x,y_{j+\hf}^-) -\widehat{f u^y}_{j-\hf} 
    \vphi_h(x,y_{j-\hf}^+)dx,\\
  \end{split} 
\end{equation} 
\begin{equation}\label{sch-2dnum-2} 
  \begin{split}
    \niintji {u}^x_h{\psi}^x_h+u^y_h\psi^y_h dxdy =& -\niintji \xi_h
    (\partial_x\psi^x_h+\partial_y\psi^y_h) dxdy \\ 
    & + \nintj \widehat{\xi}_{i+\hf}\psi^x_h(x_{i+\hf}^-,y) -
    \widehat{\xi}_{i-\hf}\psi_h^y(x_{i-\hf}^+,y) dy \\ 
    & +\ninti \widehat{\xi}_{j+\hf}\psi^y_h(x,y_{j+\hf}^-) -
    \widehat{\xi}_{j-\hf}\psi^y_h(x,y_{j-\hf}^+)dx.  
  \end{split} 
\end{equation}
Here, when the interaction potential $W$ is smooth, we set 
\begin{equation*}
    \xi_h = \xi_h(\rho_h,x,y) = H'_h + V + \nintJ\nintI
      W(x-\xx,y-\yy)\rho_h(\xx,\yy) d\xx d\yy.  
\end{equation*} 
While for non-smooth $W$, the quadrature may not achieve sufficient accuracy.
Hence the exact integration is applied  
\begin{equation*} 
  \xi_h = \xi_h(\rho_h,x,y) = H'_h + V + \int_J\int_I
  W(x-\xx,y-\yy)\rho_h(\xx,\yy) d\xx d\yy.  
\end{equation*} 
The numerical fluxes are chosen in the following way, 
\begin{eqnarray*}
  \widehat{\xi}_{i+\hf} =
    \widehat{\xi}_{i+\hf}(y)&=&\frac{1}{2}\left(\xi_h(x_{i+\hf}^+,y) +
    \xi(x_{i+\hf}^-,y)\right),\\ 
    \widehat{\xi}_{j+\hf} = \widehat{\xi}_{j+\hf}(x)&=&\frac{1}{2}
    \left(\xi_h(x,y_{j+\hf}^+) + \xi_h(x,y_{j+\hf}^-)\right),\\ 
    \widehat{fu^x}_{i+\hf} = \widehat{fu^x}_{i+\hf}(y) &=& \frac{1}{2}\left(
    (f_hu_h)(x_{i+\hf}^+,y) + (f_hu_h)(x_{i+\hf}^-,y)\right)+\frac{\alpha^x}{2}
    \left(g_h(x_{i+\hf}^+,y) -g_h(x_{i+\hf}^-,y)\right),\\ 
    \alpha^x_{i+\hf}=\alpha^x_{i+\hf}(y) &=& \max\{|u_h(x_{i+\hf}^+,y)|,
    |u_h(x_{i+\hf}^-,y)|\},\\ 
  \widehat{fu^y}_{j+\hf} = \widehat{fu^y}_{j+\hf}(x) &=& \frac{1}{2}\left( 
    (f_hu_h)(x,y_{j+\hf}^+) + (f_hu_h)(x,y_{j+\hf}^-) \right)+
  \frac{\alpha^y}{2}\left(g_h(x,y_{j+\hf}^+) -g_h(x,y_{j+\hf}^-)
  \right),\\ 
  \alpha^y_{j+\hf}=\alpha^y_{j+\hf}(x) &=& \max\{|u_h(x,y_{j+\hf}^+)|,
  |u_h(x,y_{j+\hf}^-)|\}, 
\end{eqnarray*} 
with $g_h(x,y) = g(\rho_h(x,y))$, where $g(\rho) = f(\rho)$ if $f$ is increasing and
$g(\rho) = C\rho$ if $\frac{f(\rho)}{\rho}\leq C$. 

For smooth $W$, one can obtain an entropy inequality as we have done for one
dimensional problems.
\begin{THM} 
For smooth interaction kernel $W$, assume that the semi-discrete 
scheme defined by \eqref{sch-2dnum-1} and \eqref{sch-2dnum-2} has a solution, then 
it satisfies the following entropy inequality.
  \begin{equation}\label{eq-2ddisc-ei} 
    \frac{d}{dt}\tilde{E} \leq -\tilde{I},
  \end{equation} 
  where 
  \begin{equation*} 
    \tilde{E} = \nintJ\nintI H(\rho_h) dxdy + \nintJ\nintI V\rho_h dxdy +
    \frac{1}{2}\nintJ\nintI\nintJ\nintI W(x-\xx,y-\yy)\rho_h(x,y)\rho_h(\xx,\yy)
    d\xx d\yy dx dy
  \end{equation*} 
  is the discrete entropy and 
  \begin{equation*} 
    \tilde{I} =  \nintJ\nintI f_h|\mathbf{u}_h|^2 dxdy  
  \end{equation*} is the associated discrete entropy dissipation.  
  {Moreover, if $\alpha^x$ and $\alpha^y$ are all positive and $\rho_h\geq 0$ is a stationary state of the semi-discrete scheme, then $\rho_h$ is continuous and the piecewise polynomial interpolation of $\xi_h$ is constant in each connected component of the support of $\rho_h$.}
\end{THM}
\begin{proof} 
  We will focus on the entropy-entropy dissipation relationship 
    and the proof of the second part of the theorem is omitted.
  Using the symmetry of $W$, we have 
  \begin{equation*} 
    \begin{split}
      &\frac{d}{dt}\frac{1}{2}\nintJ\nintI\nintJ\nintI
      W(x-\xx,y-\yy)\rho_h(x,y)\rho_h(\xx,\yy) d\xx d\yy dx dy\\ 
      =& \nintJ\nintI \partial_t \rho_h(x,y)\left( \nintJ\nintI 
        W(x-\xx,y-\yy)\rho_h(\xx,\yy)d\xx d\yy\right) dxdy.  
    \end{split} 
  \end{equation*} 
  Hence, 
  \begin{equation*}
    \begin{split} 
      \frac{d}{dt}\tilde{E} =&\nintJ\nintI\partial_t\rho_h(x,y)
      \left(H'_h + V +\nintJ\nintI W(x-\xx,y-\yy)\rho_h(\xx,\yy)d\xx d\yy \right) dxdy\\
      =&  \nintJ\nintI\partial_t\rho_h \xi_h dxdy = \sum_{i,j}\niintji\partial_t \rho_h
          \mI(\xi_h) dxdy\\ 
      =& \sum_{i,j} \left(-\nintj\left(\ninti\mI(f_hu_h^x)\partial_x\mI(\xi_h)dx\right)dy
        -\ninti\left(\nintj\mI(f_hu_h^y)\partial_y\mI(\xi_h)dy\right)dx\right. \\ 
        &+ \nintj \widehat{f u^x}_{i+\hf}\xi_h(x_{i+\hf}^-,y)-\widehat{fu^x}_{i-\hf}
          \xi_h(x_{i-\hf}^+,y) dy \\
      &\left.+ \ninti\widehat{fu^y}_{j+\hf}\xi_h(x,y_{j+\hf}^-) -
      \widehat{fu^y}_{j-\hf}\xi_h(x,y_{j-\hf}^+) dx \right).
    \end{split}
  \end{equation*} 
  For fixed $y$, $\mI(f_hu^x_h)\partial_x\mI(\xi_h)$ is a polynomial of degree
  $2k-1$ with respect to $x$. Hence the Gauss-Lobatto quadrature with $k+1$
  nodes is exact.  We replace the quadrature with the exact integral, integrate
  by parts and then change back to the quadrature. The same argument also
  applies to the second integral. One can then obtain,
  \begin{equation*} 
    \begin{split} 
      \frac{d}{dt}\tilde{E}=& \sum_{i,j}\left(\nintj\left(\ninti
    \partial_x\mI(f_hu_h^x)\mI(\xi_h)dx\right)dy+\ninti\left(\nintj
  \partial_y\mI(f_hu_h^y)\mI(\xi_h)dy\right)dx\right. \\ 
  &- \nintj (f_hu_h^x \xi_h)(x_{i+\hf}^-,y)-(f_hu_h^x\xi_h)(x_{i-\hf}^+,y) dy \\
    &- \ninti(f_hu_h^y\xi_h)(x,y_{j+\hf}^-) - (f_hu_h^y\xi_h)(x,y_{j-\hf}^+) dx\\ 
    & + \nintj\widehat{f u^x}_{i+\hf}\xi_h(x_{i+\hf}^-,y)-\widehat{fu^x}_{i-\hf}
    \xi_h(x_{i-\hf}^+,y) dy\\
    & \left. + \ninti\widehat{fu^y}_{j+\hf}\xi_h(x,y_{j+\hf}^-) -
  \widehat{fu^y}_{j-\hf}\xi_h(x,y_{j-\hf}^+) dx \right).\\
    \end{split} 
  \end{equation*} 
  Use the scheme \eqref{sch-2dnum-2} one can get
  \begin{equation*} 
    \begin{split} 
      \frac{d}{dt}\tilde{E}=& -\nintJ\nintI f|\mathbf{u}_h|^2 dxdy 
      + \sum_{i,j}\left(\nintj \widehat{\xi}_{i+\hf} (f_hu_h^x)(x_{i+\hf}^-,y)
        - \widehat{\xi}_{i-\hf}(f_hu_h^x)(x_{i-\hf}^+,y)dy\right.\\ 
        &+ \ninti \widehat{\xi}_{j+\hf}(f_hu_h^y)(x,y_{j+\hf}^-) -
      \widehat{\xi}_{j-\hf}(f_hu_h^y)(x,y_{j-\hf}^+) dx\\ 
        &- \nintj (f_hu_h^x \xi_h)(x_{i+\hf}^-,y)-(f_hu_h^x\xi_h)(x_{i-\hf}^+,y) dy \\
        &- \ninti(f_hu_h^y\xi_h)(x,y_{j+\hf}^-) - (f_hu_h^y\xi_h)(x,y_{j-\hf}^+) dx\\ 
        &+ \nintj\widehat{f u^x}_{i+\hf} \xi_h(x_{i+\hf}^-,y)-
        \widehat{fu^x}_{i-\hf}\xi_h(x_{i-\hf}^+,y) dy\\
        &\left.+ \ninti \widehat{fu^y}_{j+\hf}\xi_h(x,y_{j+\hf}^-) 
         -\widehat{fu^y}_{j-\hf}\xi_h(x,y_{j-\hf}^+) dx \right)\\ 
      =&-\nintJ\nintI f|\mathbf{u}_h|^2 dxdy\\
        &- \sum_{i,j}\left(\nintj
        \frac{\alpha^x_{i+\hf}(y)}{2}\left(g_h(x_{i+\hf}^+,y)-g_h(x_{i+\hf}^-,y)\right)
        \left(\xi_h(x_{i+\hf}^+,y)-\xi_h(x_{i+\hf}^-,y)\right)dy\right.\\
        &\left.+\ninti\frac{\alpha^y_{j+\hf}(x)}{2}\left(g_h(x,y_{j+\hf}^+)
          -g_h(x,y_{j+\hf}^-)\right)\left(\xi_h(x,y_{j+\hf}^+)
        -\xi_h(x,y_{j+\hf}^-)\right)dx\right).\\
    \end{split} 
  \end{equation*} 
  {By our choices of $g$, the strict monotonicity of $H'$ and 
  the fact that $V$ and $W$ are single-valued, 
  the last term is non-positive, which gives \eqref{eq-2ddisc-ei}.  }
\end{proof}

\subsection{Time discretization and preservation of positivity}

It suffices to ensure the positivity-preserving property of the Euler forward
scheme. The high order case is automatically taken care of by SSP-RK time
discretization.

The first step is to show that, provided the solution at the current time level
is non-negative, the cell average at next time level will also be non-negative,
if a specific time step restriction is satisfied.
\begin{LEM}\label{lem-2D-posi-aver} 
  Suppose $\rho_h^n(x_i^r,y_j^s)\geq 0$, $r,s=1,\dots,k+1$. Then when 
  \begin{subequations}\label{eq-2d-posi-aver}
    \begin{align}
      \lambda^x_i &\leq \min_{r,s}
      \left\{\left(\frac{w_1\rho}{2(fu^x+\alpha^x g)}\right)(x_{i-\hf}^+,y^s_j),
        \left(\frac{w_{k+1}\rho}{2(\alpha^x g-fu^x)}\right)(x_{i+\hf}^-,y^s_j)\right\},\\
      \lambda^y_j &\leq \min_{r,s}
        \left\{\left(\frac{w_1\rho}{2(fu^y+\alpha^y g)}\right)(x^r_i,y^+_{j-\hf}),
        \left(\frac{w_{k+1}\rho}{2(\alpha^y g-fu^y)}\right)(x^r_i,y_{j+\hf}^-)\right\},
    \end{align} 
  \end{subequations} 
  the solution $(\rho_h)_{i,j}^{n+1,\pre}$ obtained from \eqref{sch-2dnum-1} satisfies
  $\left(\bar\rho_h\right)_{i,j}^{n+1,\pre}\geq 0$. Here, in the constraint of 
  $\lambda_i^x$ and $\lambda_j^x$, we formally denote by $\frac{0}{0} := +\infty$.  
\end{LEM}
\begin{proof} As before, we drop all the subscripts $h$ in this proof. The
  superscript $n$ will also be omitted for simplicity. Take $\vphi = 1$ in
  \eqref{sch-2dnum-1}, we have 
  \begin{equation*} 
    \bar{\rho}^{n+1,\pre}_{i,j} =
    \bar{\rho}_{i,j} + \frac{\tau}{h^x_ih^y_j} \nintj\widehat{f u^x}_{i+\hf}
    -\widehat{f u^x}_{i-\hf}dy +\frac{\tau}{h^x_i h^y_j} \ninti \widehat{f
    u^y}_{j+\hf}-\widehat{f u^y}_{j-\hf} dx.
  \end{equation*} 
  Note that 
  \[\bar {\rho}_{i,j} = \frac{1}{h^x_i h^y_j}\niintji\rho dx dy
    = \frac{1}{4}\sum_{r =1}^{k+1}\sum_{s = 1}^{k+1}w_{r}w_{s} \rho(x_i^{r},y_j^{s}).\] 
  Let $\lambda^x_i= \frac{\tau}{h^x_i}$ and $\lambda^y_j = \frac{\tau}{h^y_j}$, 
  then we have
  \begin{equation*} 
    \begin{aligned} 
      &\bar{\rho}_{i,j}^{n+1,\pre} =
      \frac{1}{4}\sum_{r = 1}^{k+1}\sum_{s = 1}^{k+1}  w_rw_s \rho(x^r_i,y^s_j)\\ 
      &+ \frac{\lambda^x_i }{4}\sum_{s=1}^{k+1} w_s\left((fu^x)(x^+_{i+\hf},y^s_j)
        + (fu^x)(x^-_{i+\hf},y^s_j) + \alpha^x_{i+\hf}\left(g_h(x_{i+\hf}^+,y_j^s)
      -g_h(x_{i+\hf}^-,y_j^s)\right) \right) \\ 
      &-\frac{\lambda^x_i }{4}\sum_{s=1}^{k+1} w_s\left( (fu^x)(x^+_{i-\hf},y_j^s)+
    (fu^x)(x^-_{i-\hf},y^s_j) + \alpha^x_{i-\hf}
    \left(g_h(x_{i-\hf}^+,y_j^s)-g_h(x_{i-\hf}^-,y_j^s)\right) \right) \\ 
    &+\frac{\lambda^y_j }{4}\sum_{r=1}^{k+1} w_r\left( (fu^y)(x^r_i,y_{j+\hf}^+)+
    (fu^y)(x^r_i,y^-_{j+\hf}) + \alpha^y_{j+\hf}
    \left(g_h(x^r_i,y_{j+\hf}^+)-g_h(x^r_i,y_{j+\hf}^-)\right) \right) \\ 
    &-\frac{\lambda^y_j }{4}\sum_{r=1}^{k+1} w_r\left( (fu^y)(x^r_i,y^+_{j-\hf})+
    (fu^y)(x^r_i,y^-_{j-\hf}) + \alpha^y_{j-\hf}
  \left(g_h(x_i^r,y_{j-\hf}^+)-g_h(x_i^r,y_{j-\hf}^-)\right) \right) \\ 
    & \geq\frac{1}{4}\sum_{r = 2}^k\sum_{s = 2}^k w_rw_s\rho(x^r_i,y_j^s)\\ 
    &+ \sum_{s =1}^{k+1}w_s\left(\frac{w_{k+1}}{8}\rho(x^-_{i+\hf},y_j^s) +
  \frac{\lambda^x_i}{4}\left( (fu^x)(x^-_{i+\hf},y^s_j) -\alpha^x_{i+\hf}
  g_h(x_{i+\hf}^-,y_j^s) \right) \right)\\ 
    &+ \sum_{s=1}^{k+1}w_s\left(\frac{w_1}{8}\rho(x^+_{i-\hf},y_j^s)- 
    \frac{\lambda^x_i }{4}\left((fu^x)(x^+_{i-\hf},y_j^s) + \alpha^x_{i-\hf}
    g_h(x_{i-\hf}^+,y_j^s)\right)\right)\\ 
    &+\sum_{r=1}^{k+1}w_r\left(\frac{w_{k+1}}{8}\rho(x_i^r,y^-_{j+\hf}) +
    \frac{\lambda^y_j}{4}\left((fu^y)(x^r_i,y^-_{j+\hf}) -
    \alpha^y_{j+\hf}g_h(x^r_i,y_{j+\hf}^-) \right)\right) \\ 
    &+\sum_{r=1}^{k+1} w_r\left(\frac{w_1}{8}\rho(x^r_i,y^+_{j-\hf}) 
      - \frac{\lambda^y_j }{4}\left((fu^y)(x^r_i,y^+_{j-\hf})
        + \alpha^y_{j-\hf}g_h(x_i^r,y_{j-\hf}^+)\right)\right)\\
    &+ \frac{\lambda^x_i }{4}\sum_{s=1}^{k+1}w_s\left( (fu^x)(x^+_{i+\hf},y^s_j)+ 
      \alpha^x_{i+\hf}g_h(x_{i+\hf}^+,y_j^s)\right)\\
      &-\frac{\lambda^x_i}{4}\sum_{s=1}^{k+1}
      w_s\left((fu^x)(x^-_{i-\hf},y^s_j)-\alpha^x_{i-\hf} g_h(x_{i-\hf}^-,y_j^s)\right) \\
    &+ \frac{\lambda^y_j }{4}\sum_{r=1}^{k+1} w_r\left((fu^y)(x^r_i,y_{j+\hf}^+)
    + \alpha^y_{j+\hf} g_h(x^r_i,y_{j+\hf}^+)\right) \\
    &-\frac{\lambda^y_j }{4}\sum_{r=1}^{k+1} w_r\left( (fu^y)(x^r_i,y^-_{j-\hf}) -
    \alpha^y_{j-\hf}g_h(x_i^r,y_{j-\hf}^-) \right). \\ 
    \end{aligned} 
  \end{equation*}

The first term is automatically non-negative, since the weights $w_r,w_s\geq 0$
and the nodal values $\rho(x_i^r,y_j^s)\geq 0$. The positivity of the last four
terms is guaranteed by our choice of $\alpha$ and $g$. One only needs
\eqref{eq-2d-posi-aver} to ensure the second and the third term to be
non-negative. And as before, one can check the convention $\frac{0}{0}=+\infty$
does make sense. Hence $\bar \rho^{n+1}_{i,j} \geq 0$ under the prescribed time
step constraint. 
\end{proof}

Then, as we have done in the one dimensional case, a scaling limiter is applied
to sure the numerical polynomial solution takes non-negative values at the
quadrature points. Hence the assumption in Lemma \ref{lem-2D-posi-aver} is met
and the fully discretized scheme is positivity-preserving.
\begin{THM} 
  Let 
  \[\rho_h^{n+1}(x_i^r,y_j^s) = (\bar \rho_h)_{i,j}^{n+1,\pre}
    +\theta_{i,j}\left(\rho_h^{n+1,\pre}(x_i^r,y_j^s)- (\bar \rho_h)_{i,j}^{n+1,\pre}\right),
    \qquad\forall r,s = 1,\dots,k+1,\] 
    with $\theta_{i,j}=\min\{\frac{ (\bar \rho_h)_{i,j}^{n+1,\pre} }
      {(\bar \rho_h)_{i,j}^{n+1,\pre}-m_{i,j}},1\}$ and $m_{i,j} 
      = \min\{\rho_h^{n+1,\pre}(x_i^r,y_j^s)\}_{r,s =1}^{k+1}$. 
      Then we have $\rho_h^{n+1}(x_i^r,y_j^s)\geq0,\forall r,s =
    1,\dots,k+1$, $\bar{\rho}_h^{n+1} = \bar{\rho}_h^{n+1,\pre}$. Hence the 
    resulting fully discretized scheme using Euler forward or SSP-RK time 
    discretization preserves the non-negativity of the solution, if
    \begin{equation*} 
      \begin{split}
        \tau\leq \min_{i,j}\{h^x_i,h^y_j\}\cdot \min_{i,j,s,r}
      \left\{
        \left(\frac{w_1\rho}{2(fu^x+\alpha^x g)}\right)(x_{i-\hf}^+,y^s_j),
        \left(\frac{w_{k+1}\rho}{2(\alpha^x g-fu^x)}\right)(x^-_{i+\hf},y^s_j)\right.\\
        \left.\left(\frac{w_1\rho}{2(fu^y+\alpha^yg)}\right)(x^r_i,y^+_{j-\hf}),
        \left(\frac{w_{k+1}\rho}{2(\alpha^y g-fu^y)}\right)(x^r_i,y_{j+\hf}^-)
      \right\}.
      \end{split} 
    \end{equation*} 
\end{THM}

%****************************************************************************************
\section{One dimensional numerical tests} \label{sec-num-result-1d}
\setcounter{equation}{0} 
\setcounter{figure}{0} 
\setcounter{table}{0}
%****************************************************************************************
%                               One dimensional numerical tests
%****************************************************************************************
\subsection{Accuracy tests} 
In this part, we examine the accuracy of the numerical schemes with $P^1$, $P^2$, 
$P^3$ and $P^4$ elements. The error is measured in the discrete norms.  
\begin{subequations} 
  \begin{align*} 
    e_{L^1} &= \sum_i\ninti \left|u_h(x,t) - u(x,t)\right| dx,\\ 
    e_{L^2} &= \sqrt{\sum_i\ninti \left|u_h(x,t) - u(x,t)\right|^2 dx},\\ 
    e_{L^\infty} &= \max_{x\in \{x_i^r\}_{i,r}}\left|u_h(x,t) - u(x,t)\right|.  
  \end{align*} 
\end{subequations}
%-----------------------------------------------------------------------------------------
\begin{examp}[advection equation]\label{examp:adv} 
  The first numerical test is done for the linear advection equation 
  \begin{equation*} 
    \left\{
      \begin{aligned} 
        \partial_t \rho &= \partial_{x}\rho ,\quad x\in[-\pi,\pi],\\ 
        \rho(x,0) &= 1 + \sin(x).  
      \end{aligned}
    \right.
  \end{equation*} 
  The problem has an exact solution $u(x,t) = 1+\sin(x+t)$.  In this test,
  $f(\rho) = \rho$, $H'(\rho) = 0$, $V(x) = x$ and $W(x) = 0$.  To be consistent
  at the boundaries, one needs to manually impose $\widehat{\xi}(\pi) = \pi$ and
  $\widehat{\xi}(-\pi) = -\pi$. (This will gives $u \equiv 1$ and the scheme is
  equivalent to the usual upwinding DG method with a \emph{mass lumping}
  treatment.) We compute up to $t = 2$ and the time step is $\tau = 0.02h^2$.

  Due to our choice of the initial condition, the solution has point vacuum and
  the numerical solution may become negative in its neighborhood. We perform
  numerical tests without and with the positivity-preserving limiter and the
  results are listed in Table \ref{tab:adv_nolimiter} and Table
  \ref{tab:adv_limiter} respectively. As one can see, without the limiter, the
  convergence rate is optimal. The rate degenerates a little bit for 
the $P^4$ scheme when one applies the limiter.

\begin{table}[h!] 
  \centering 
  \begin{tabular}{c|c|c|c|c|c|c|c} 
    \hline   
    k &N&$L^1$ error& order&$L^2$ error& order&$L^\infty$ error& order\\
    \hline
    1&20 &0.155489         &  -      & 0.689292E-01  &     -  &0.416916E-01 & - \\ 
     &40 &0.403867E-01     &  1.94   & 0.179328E-01  &  1.94  &0.106198E-01 & 1.97  \\
     &80 &0.102281E-01     &  1.98   & 0.453598E-02  &  1.98  &0.265698E-02 & 2.00  \\ 
     &160&0.256686E-02     &  1.99   & 0.113801E-02  &  1.99  &0.663269E-03 & 2.00  \\ 
     \hline
    2&20 &0.183679E-02     &  -      & 0.104224E-02  &     -  &0.124147E-02&  -    \\  
     &40 &0.222515E-03     &  3.05   & 0.130558E-03  &  3.00  &0.158800E-03&  2.97 \\  
     &80 &0.273812E-04     &  3.02   & 0.163282E-04  &  3.00  &0.200245E-04&  2.99 \\  
     &160&0.339363E-05     &  3.01   & 0.204129E-05  &  3.00  &0.251331E-05&  2.99 \\  
     \hline
    3&20 &0.299466E-04     &  -      & 0.176257E-04  &     -  &0.270453E-04&  -    \\   
     &40 &0.187719E-05     &  4.00   & 0.110354E-05  &  4.00  &0.170007E-05&  3.99 \\   
     &80 &0.117691E-06     &  4.00   & 0.689895E-07  &  4.00  &0.106066E-06&  4.00 \\   
     &160&0.736323E-08     &  4.00   & 0.431213E-08  &  4.00  &0.662179E-08&  4.00 \\   
     \hline
    4&20 &0.450982E-06     &  -      & 0.252754E-06  &     -  &0.429430E-06&  -    \\   
     &40 &0.133008E-07     &  5.08   & 0.798519E-08  &  4.98  &0.143225E-07&  4.91 \\   
     &80 &0.415333E-09     &  5.00   & 0.246970E-09  &  5.01  &0.444279E-09&  5.01 \\   
     &160&0.129717E-10     &  5.00   & 0.771945E-11  &  5.00  &0.138960E-10&  5.00 \\   
     \hline
  \end{tabular} 
  \caption{Accuracy test of the linear advection equation in Example \ref{examp:adv}: 
  without limiters.} 
  \label{tab:adv_nolimiter} 
\end{table}
\begin{table}[h!] 
  \centering 
  \begin{tabular}{c|c|c|c|c|c|c|c} 
    \hline   
    k &N&$L^1$ error& order&$L^2$ error& order&$L^\infty$ error& order\\ \hline
    1&20 & 0.149399      & -     &  0.667930E-01  & -    & 0.433314E-01  & -\\ 
     &40 & 0.400851E-01  & 1.90  &  0.181381E-01  & 1.88 & 0.136753E-01  &1.66   \\ 
     &80 & 0.104653E-01  & 1.94  &  0.473425E-02  & 1.94 & 0.514818E-02  & 1.41   \\ 
     &160& 0.268565E-02  & 1.96  &  0.122685E-02  & 1.95 & 0.176080E-02  & 1.55   \\ 
    \hline 
    2&20 & 0.183523E-02  & -     &  0.104831E-02  & -    & 0.124144E-02  & -      \\ 
     &40 & 0.223090E-03  & 3.04  &  0.130674E-03  & 3.00 & 0.158800E-03  & 2.97   \\ 
     &80 & 0.274294E-04  & 3.02  &  0.163317E-04  & 3.00 & 0.200245E-04  & 2.99   \\ 
     &160& 0.339506E-05  & 3.01  &  0.204138E-05  & 3.00 & 0.251331E-05  & 2.99   \\ 
    \hline 
    3&20 & 0.313613E-04  & -     &  0.188466E-04  & -    & 0.359345E-04  & -      \\
     &40 & 0.199045E-05  & 3.98  &  0.117327E-05  & 4.01 & 0.182708E-05  & 4.30\\ 
     &80 & 0.121686E-06  & 4.03  &  0.719577E-07  & 4.02 & 0.162878E-06  & 3.49   \\ 
     &160& 0.759071E-08  & 4.00  &  0.446279E-08  & 4.01 & 0.945333E-08  & 4.11   \\ 
    \hline 
    4&20 & 0.166111E-05  & -     &  0.160456E-05  & -    & 0.285453E-05  & -      \\ 
     &40 & 0.583064E-07  & 4.83  &  0.758033E-07  & 4.40 & 0.204428E-06  & 3.80   \\ 
     &80 & 0.199636E-08  & 4.87  &  0.359090E-08  & 4.40 & 0.138462E-07  & 3.88   \\ 
     &160& 0.707240E-10  & 4.82  &  0.170196E-09  & 4.40 & 0.903761E-09  & 3.94   \\ 
    \hline 
  \end{tabular} 
  \caption{Accuracy test of the linear advection equation in Example \ref{examp:adv}: 
  with limiters.}
  \label{tab:adv_limiter} 
  \end{table} 
\end{examp}
%---------------------------------------------------------------------------------------------
\begin{examp}[heat equation]\label{examp:heat} 
  We then examine the heat equation, 
  \begin{equation*} 
    \left\{ 
      \begin{aligned} 
        \partial_t \rho &= \partial_{xx}\rho ,\quad x\in [-\pi,\pi],\\ 
        \rho(x,0) &= 2 + \sin(x),
      \end{aligned}\right.  
  \end{equation*} 
  with periodic boundary conditions. The exact solution to the problem is 
  $\rho(x,t) = 2 + e^{-t}\sin(x)$. The decomposition of the equation into the 
  desired form is not unique. Let us consider two test cases, \\ 
  (i) $f(\rho) = \rho$, $H'(\rho) = \log(\rho)$ and $V(x) = W(x) = 0$, \\ 
  (ii) $f(\rho) = \sqrt{\rho}$, $H'(\rho) = 2\sqrt{\rho}$ and $V(x) = W(x) = 0$. \\ 
  Note that for both of the cases, the schemes are nonlinear, although the
  original problem is linear. We use the time step $\tau = 0.01h^2$ 
to compute
  to $t = 2$.  Error tables are given in Table \ref{tab:heat_i} and Table
  \ref{tab:heat_ii} respectively. According to our numerical results, we see
  that different choices of decomposition lead to negligible difference. For
  both of the tests, $P^2$ and $P^4$ schemes are of the optimal rate of
  convergence, but the order for $P^1$ and $P^3$ schemes seems 
to be reduced.

\begin{table}[h!] 
  \centering 
  \begin{tabular}{c|c|c|c|c|c|c|c} 
    \hline   
    k&N&$L^1$ error& order&$L^2$ error& order&$L^\infty$ error& order\\ 
    \hline 
    1&20 & 0.795669E-02&  -   & 0.369447E-02  &-      & 0.228808E-02& -     \\
     &40 & 0.200183E-02&  1.99& 0.988037E-03  &1.90   & 0.664459E-03& 1.78  \\
     &80 & 0.552074E-03&  1.86& 0.283256E-03  &1.80   & 0.202063E-03& 1.72  \\
     &160& 0.172193E-03&  1.68& 0.855650E-04  &1.73   & 0.622412E-04& 1.70  \\
     &320& 0.538010E-04&  1.68& 0.259228E-04  &1.72   & 0.187767E-04& 1.73  \\
    \hline 
    2&20 & 0.153364E-03&  -   & 0.935049E-04  &-     & 0.901032E-04& - \\  
     &40 & 0.167874E-04&  3.19& 0.113109E-04  &3.05  & 0.110833E-04& 3.02\\  
     &80 & 0.195595E-05&  3.10& 0.140286E-05  &3.01  & 0.138186E-05& 3.00\\  
     &160& 0.235834E-06&  3.05& 0.175062E-06  &3.00  & 0.172667E-06& 3.00\\  
     &320& 0.289492E-07&  3.02& 0.218767E-07  &3.00  & 0.215818E-07& 3.00\\  
    \hline 
    3&20 & 0.162173E-04&  -   & 0.780319E-05  &-     & 0.789576E-05& -     \\  
     &40 & 0.180537E-05&  3.17& 0.867447E-06  &3.17  & 0.877086E-06& 3.17  \\  
     &80 & 0.185055E-06&  3.29& 0.892168E-07  &3.28  & 0.909961E-07& 3.27  \\  
     &160& 0.171294E-07&  3.43& 0.833148E-08  &3.42  & 0.865799E-08& 3.39  \\  
     &320& 0.142478E-08&  3.59& 0.705413E-09  &3.56  & 0.756372E-09& 3.52  \\  
    \hline 
    4&20 & 0.357641E-07&  -   & 0.237294E-07  &-     & 0.410404E-07& -     \\  
     &40 & 0.104036E-08&  5.10& 0.720811E-09  &5.04  & 0.125451E-08& 5.03  \\  
     &80 & 0.315216E-10&  5.04& 0.223737E-10  &5.01  & 0.390045E-10& 5.01  \\  
     &160& 0.971067E-12&  5.02& 0.698093E-12  &5.00  & 0.121750E-11& 5.00  \\  
     &320& 0.301387E-13&  5.01& 0.218084E-13  &5.00  & 0.380404E-13& 5.00  \\  
    \hline 
  \end{tabular} 
  \caption{Accuracy test of the heat equation in Example \ref{examp:heat}: (i) 
    $\partial_t\rho = \partial_x\left(\rho\partial_x\log(\rho)\right)$.} 
    \label{tab:heat_i}
\end{table} 
\begin{table}[h!] 
  \centering 
  \begin{tabular}{c|c|c|c|c|c|c|c} 
    \hline
    k  &N&$L^1$ error& order&$L^2$ error& order&$L^\infty$ error& order\\ 
    \hline
    1 &20 &   0.842713E-02 &  -    &  0.378282E-02 & -    & 0.223942E-02  &-\\ 
      &40 &   0.210694E-02 &  2.01 &  0.967781E-03 & 1.97 & 0.606923E-03&1.88\\ 
      &80 &   0.531435E-03 &  2.00 &  0.258302E-03 & 1.92 & 0.174263E-03&1.80\\ 
      &160&   0.143005E-03 &  1.89 &  0.733990E-04 & 1.83 & 0.522090E-04&1.74\\ 
      &320&   0.439029E-04 &  1.70 &  0.219496E-04 & 1.74 & 0.159240E-04&1.71\\ 
    \hline 
    2 &20 &   0.157192E-03 &  -    &  0.939409E-04 & -    &0.892857E-04  &-   \\ 
      &40 &   0.169943E-04 &  3.21 &  0.113196E-04 & 3.05 &0.110079E-04  &3.02\\ 
      &80 &   0.197008E-05 &  3.11 &  0.140217E-05 & 3.01 &0.136949E-05  &3.01\\ 
      &160&   0.236877E-06 &  3.06 &  0.174896E-06 & 3.00 &0.171085E-06  &3.00\\ 
      &320&   0.290358E-07 &  3.03 &  0.218519E-07 & 3.00 &0.213811E-07  &3.00\\ 
    \hline 
    3 &20 &    0.174043E-04 &  -    &  0.830168E-05& -    & 0.809730E-05  &-   \\ 
      &40 &    0.204183E-05 &  3.09 &  0.970636E-06& 3.10 & 0.953802E-06  &3.09\\ 
      &80 &    0.226482E-06 &  3.17 &  0.107547E-06& 3.17 & 0.105447E-06  &3.18\\ 
      &160&    0.231941E-07 &  3.29 &  0.110068E-07& 3.29 & 0.107849E-07  &3.29\\ 
      &320&    0.214673E-08 &  3.43 &  0.101820E-08& 3.43 & 0.998098E-09  &3.43\\ 
    \hline 
    4 &20 &    0.352777E-07 &  -    &0.218885E-07 & -    & 0.337766E-07  &-   \\ 
      &40 &    0.103272E-08 &  5.09 &0.668083E-09 & 5.03 & 0.105025E-08  &5.01\\ 
      &80 &    0.313767E-10 &  5.04 &0.207580E-10 & 5.01 & 0.326896E-10  &5.01\\ 
      &160&    0.967915E-12 &  5.02 &0.647801E-12 & 5.00 & 0.102174E-11  &5.00\\ 
      &320&    0.300611E-13 &  5.01 &0.202376E-13 & 5.00 & 0.319234E-13  &5.00\\ 
    \hline 
  \end{tabular}
  \caption{Accuracy test of the heat equation in Example \ref{examp:heat}: (ii)
  $\partial_t\rho = \partial_x\left(\sqrt{\rho}\partial_x(2\sqrt{\rho})\right)$.}
  \label{tab:heat_ii} 
  \end{table} 
\end{examp} 
%----------------------------------------------------------------------------------
\begin{examp}[evolution equation with interaction potentials]\label{examp:kernel}
  Our final tests are designed for problems with interaction potentials.  
  \begin{equation}\label{eq-kernel} 
    \left\{
      \begin{aligned} 
        \partial_t\rho &= \partial_x\left(\rho\partial_x (W\ast
        \rho)\right),\quad x\in[-1,1],\\ 
        \rho(x,0) &= \left(\frac{e^{-\frac{x^2}{0.1}}}{\sqrt{0.1\pi}}\right)^4.
      \end{aligned}\right.  
  \end{equation} 
  Periodic boundary conditions are applied for the problem. We consider both the
  smooth case $W(x) =0.2\frac{e^{-\frac{x^2}{0.1}}}{\sqrt{0.1\pi}}$ and the
  nonsmooth case $W(x) = \max\{0.2-|x|,0\}$. The convolution integrals are
  evaluated by quadrature and exact integration respectively. We compute to $t =
  0.2$ with $\tau = 0.4h^2$ and use the numerical solution with $P^4$ elements
  on $N = 1280$ mesh as the reference solution to evaluate the accuracy. We do
  not impose the limiter in the tests. The order of accuracy seems to be
  optimal for odd $k$, while the order degenerates for even $k$. See Table 
  \ref{tab:kernel_smooth} and Table \ref{tab:nonsmooth-kernel}.

\begin{table}[h!] 
  \centering 
  \begin{tabular}{c|c|c|c|c|c|c|c} 
    \hline k  &N&$L^1$ error& order&$L^2$ error& order&$L^\infty$ error& order\\ 
    \hline 
    1&40 & 0.859828E-01 &    -  &  0.987370E-01 &   -    & 0.194420    &    -  \\ 
     &80 & 0.286579E-01 &  1.59 &  0.355164E-01 & 1.48   & 0.771407E-01&  1.33 \\
     &160& 0.812899E-02 &  1.82 &  0.106417E-01 & 1.74   & 0.243226E-01&  1.67 \\
     &320& 0.212984E-02 &  1.93 &  0.284231E-02 & 1.90   & 0.686847E-02&  1.82 \\
     &640& 0.539424E-03 &  1.98 &  0.725347E-03 & 1.97   & 0.178994E-02&  1.94 \\
    \hline 
    2&40 & 0.261541E-01 &    -  &  0.689950E-01 &   -    & 0.458242    &  -  \\ 
     &80 & 0.403687E-02 &  2.70 &  0.130854E-01 & 2.40   & 0.122350    &  1.91 \\   
     &160& 0.668047E-03 &  2.60 &  0.236858E-02 & 2.47   & 0.310993E-01&  1.98 \\   
     &320& 0.127115E-03 &  2.39 &  0.426574E-03 & 2.47   & 0.780721E-02&  1.99 \\   
    % &640& 0.271767E-04 &  2.23 &  0.776973E-04 & 2.46   & 0.195379E-02&2.00 \\   
    \hline 
    3&40 & 0.767936E-03 &    -  & 0.125522E-02 &   -     & 0.100354E-01&    -  \\ 
     &80 & 0.448571E-04 &  4.10 & 0.671569E-04 & 4.22    & 0.649938E-03&  3.95 \\   
     &160& 0.245043E-05 &  4.19 & 0.357006E-05 & 4.23    & 0.409350E-04&  3.99 \\   
     &320& 0.131245E-06 &  4.22 & 0.193372E-06 & 4.21    & 0.250925E-05&  4.03 \\   
     &640& 0.725657E-08 &  4.18 & 0.106902E-07 & 4.18    & 0.101879E-06&  4.62 \\   
    \hline 
    4&40 & 0.826595E-04 &    -  &  0.236938E-03 &   -    & 0.283191E-02&    -  \\ 
     &80 & 0.301748E-05 &  4.78 &  0.114572E-04 & 4.37   & 0.195543E-03&  3.86 \\
     &160& 0.110995E-06 &  4.76 &  0.518909E-06 & 4.46   & 0.125786E-04&  3.96 \\
     &320& 0.433154E-08 &  4.68 &  0.240724E-07 & 4.43   & 0.846605E-06&  3.89 \\   
%  &640& 0.227755E-09 &  4.25 &  0.229180E-08 & 3.39   & 0.115532E-06&  2.87 \\   
    \hline 
  \end{tabular} 
  \caption{Accuracy test for Example \ref{examp:kernel} with a smooth 
  kernel $W(x) =0.2\frac{e^{-\frac{x^2}{0.1}}}{\sqrt{0.1\pi}}$.} 
  \label{tab:kernel_smooth} 
\end{table}
\begin{table}[h!] 
  \centering 
  \begin{tabular}{c|c|c|c|c|c|c|c} 
    \hline 
    k  &N&$L^1$error& order&$L^2$ error& order&$L^\infty$ error& order\\ 
    \hline 
    1&40 &0.107219     &    - &     0.129220&   - &0.276379    &- \\ 
     &80 &0.379424E-01 & 1.50 & 0.509462E-01&1.34 &0.146170    &0.92\\ 
     &160&0.113287E-01 & 1.74 & 0.167052E-01&1.61 &0.529961E-01&1.46\\ 
     &320&0.306399E-02 & 1.89 & 0.480518E-02&1.80 &0.161279E-01&1.72\\ 
     &640&0.780911E-03 & 1.97 & 0.126728E-02&1.92 &0.446252E-02&1.85\\ 
    \hline
    2&40 &0.277861E-01 &    - & 0.363441E-01&    - & 0.135178&   -\\ 
     &80 &0.506578E-02 & 2.46 & 0.703702E-02& 2.37 & 0.327120E-01&2.05\\ 
     &160&0.113483E-02 & 2.16 & 0.161949E-02& 2.12 & 0.791563E-02&2.05\\ 
     &320&0.277402E-03 & 2.03 & 0.408541E-03& 1.99 & 0.220603E-02&1.84\\ 
     &640&0.701655E-04 & 1.98 & 0.104865E-03& 1.96 & 0.582488E-03&1.92\\ 
     \hline
    3&40 &0.183849E-02 &    - & 0.347290E-02&    - & 0.240838E-01& - \\ 
     &80 &0.181121E-03 & 3.34 & 0.373509E-03& 3.22 & 0.365025E-02& 2.72\\ 
     &160&0.108761E-04 & 4.06 & 0.233375E-04& 4.00 & 0.254088E-03& 3.84\\
     &320&0.681396E-06 & 4.00 & 0.152720E-05& 3.93 & 0.185908E-04& 3.77\\
%   &640&     0.480419E-07 & 3.83 &     0.128642E-06&       3.57 &
%   0.317485E-05&       2.55\\
    \hline                                                                  
    4&40&  0.385095E-03 &    - &0.815323E-03&    - & 0.793063E-02&-\\ 
     &80 & 0.171694E-04 & 4.49 &0.244340E-04& 5.06 & 0.130440E-03& 5.93\\ 
     &160& 0.100754E-05 & 4.09 &0.169036E-05& 3.85 & 0.136589E-04& 3.26\\ 
     &320& 0.585433E-07 & 4.11 &0.133120E-06& 3.67 & 0.151993E-05& 3.17\\
%   &640&     0.892632E-08 & 2.71 &     0.411147E-07&       1.70 &
%   0.744989E-06&       1.03\\
    \hline 
 \end{tabular} 
 \caption{Accuracy test for Example \ref{examp:kernel} with a nonsmooth 
 kernel $W(x) = \max\{0.2-|x|,0\}$.} 
 \label{tab:nonsmooth-kernel} 
 \end{table}
\subsection{Fokker-Planck type equations}
%--------------------------------------------------------------------------------------
\begin{examp}[porous media equation] 
  Let us consider the porous media equation
  \[\partial_t\rho =\partial_{x}\left(\rho\partial_x\left(\frac{m}{m-1}\rho^{m-1}
  +\frac{x^2}{2}\right)\right),\]
  which is used to model the flow of a gas through a porous interface. The
  equation fits our model with $H(\rho) = \frac{1}{m-1}\rho^m$ and $V(x) =
  \frac{x^2}{2}$. The property of the equation is studied by Carrillo and
  Toscani in \cite{carrillo2000asymptotic} using an entropy approach. They
  have proved that the equation converges to a unique steady state given by a
  Barenblatt-Pattle type formula,
  \[\rho_\infty(x) = \left(C - \frac{m-1}{2m}|x|^2\right)^{\frac{1}{m-1}}.\] 
  Here the constant $C$ is determined by ensuring the mass conservation.
  Furthermore, the relative entropy $E(t|\infty) = E(\rho(t)) - E(\rho_\infty)$ decays
  exponentially, $E(t|\infty) \leq E(0|\infty)e^{-2t}$ and the rate $-2$ is
  sharp.

  We particularly choose $m=2$ in our numerical test.
  \begin{equation*}\label{porous_ivp} 
    \left\{ 
      \begin{aligned} 
        \partial_t\rho &= \partial_x(\rho\partial_x(2\rho + \frac{x^2}{2})),
        \quad x\in [-2,2],\\
        \rho(x,0) &= \max\{1-|x|,0\},
      \end{aligned}\right.  
  \end{equation*} 
  with periodic boundary conditions.  The stationary solution is 
  \[\rho_\infty = \max\left\{\left(\frac{3}{8}\right)^{\frac{2}{3}} - \frac{x^2}{4},
  0\right\}.\]
  We compute up to $t= 5$ with the number of cells $N = 40$ and the time step 
  $\tau = 0.005h^2$.
  The positivity preserving limiter keeps being invoked in the test. (If we
  manually turn off the limiter, the solution may blow up.) The profiles of the
  solution polynomials with $k = 2$ and $k = 3$ are given in Figure
  \ref{fig:porous_profile}. As one can see, the numerical solutions converge
  well to the exact steady state in the smooth region. We also provide a
  zoomed-in snapshot to exhibit the capture of singularity near $x =
  3^{\frac{1}{3}}$ in Figure \ref{fig:porous_singular}. 

  In Figure \ref{fig:porous_entropy_sym_abs} and Figure \ref{fig:porous_entropy}, we plot the entropy and the relative 
  entropy respectively. The entropy profiles for $k=2$ and $k=3$ are 
  almost identical, and they both approach to zero. 
  We then evaluate the relative entropy using that of the numerical steady state 
  as a reference. But we can only plot up to a certain time, before
the relative entropy becomes slightly negative, since the entropy 
  decay of the semi-discrete scheme may not be preserved after 
applying the 
  time discretization and the limiter.  We stop the plotting
after negative relative entropy appears, not only because it 
  can not be depicted in the logarithm scale, but also because
the unresolved tail should be 
  regarded as a discretization error. If we choose $N = 320$ while keeping 
  $\tau = 0.005h^2$, the decay will continue further.

  \begin{figure}[h!] 
    \centering
    \begin{subfigure}[h]{0.45\textwidth}
      \centering
      \includegraphics[width=\textwidth]{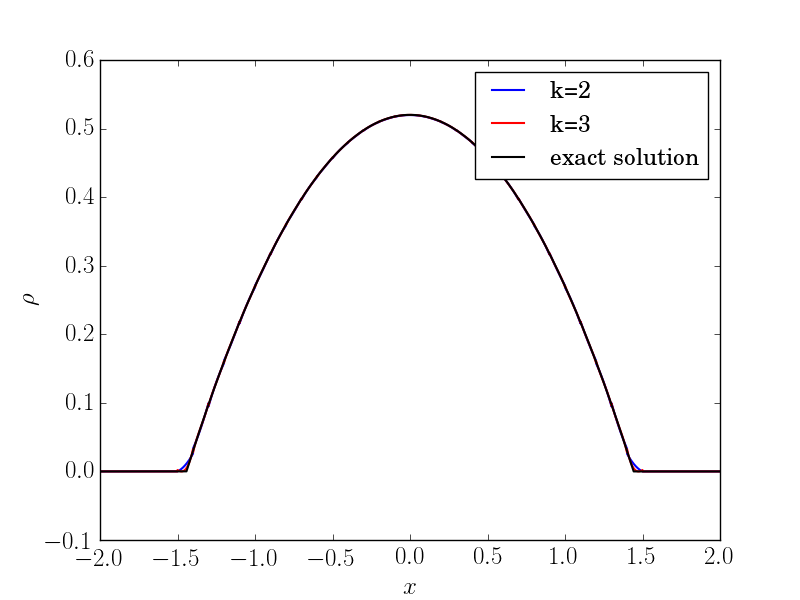} 
      \caption{Steady state.}\label{fig:porous_profile}
    \end{subfigure}
    ~
    \begin{subfigure}[h]{0.45\textwidth}
      \centering
      \includegraphics[width=\textwidth]{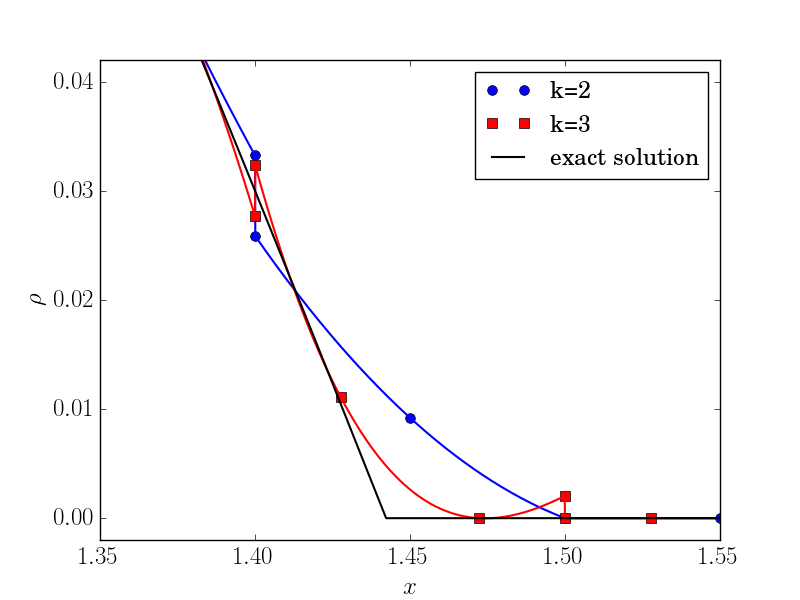} 
      \caption{Zoomed in figure near $x = 3^{\frac{1}{3}}$.}\label{fig:porous_singular}
    \end{subfigure}
    \caption{Profiles of the numerical steady states for \eqref{porous_ivp}, with
    $\rho(x,0) = \max\{1-|x|,0\}$.} 
  \end{figure} 
  \begin{figure}[h!] 
    \centering
    \begin{subfigure}[h]{0.45\textwidth}
      \centering
      \includegraphics[width=\textwidth]{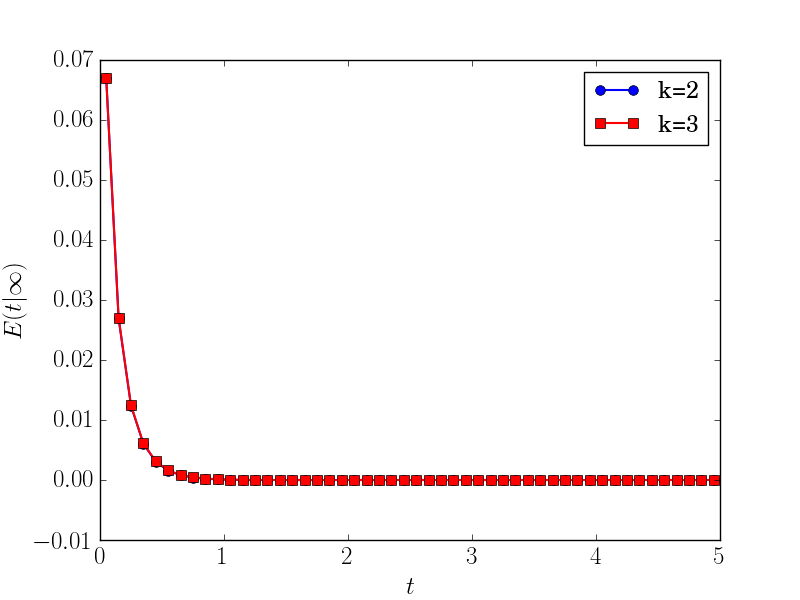} 
      \caption{{$\rho(x,0) =\max\{1-|x|,0\}$.}\label{fig:porous_entropy_sym_abs} }
    \end{subfigure}
    ~
    \begin{subfigure}[h]{0.45\textwidth}
      \centering
      \includegraphics[width=\textwidth]{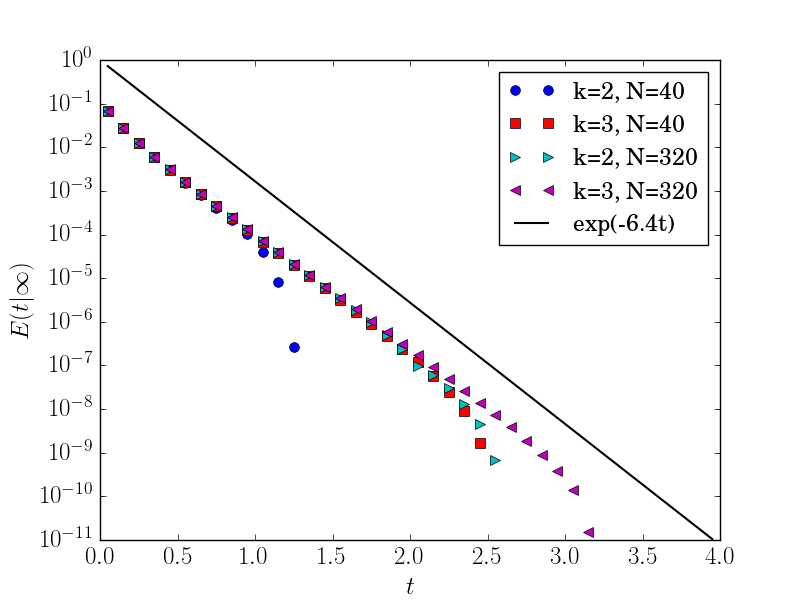} 
      \caption{{$\rho(x,0) =\max\{1-|x|,0\}$.}\label{fig:porous_entropy} }
    \end{subfigure}
    \caption{The entropy and the relative entropy for \eqref{porous_ivp}, with
    $\rho(x,0) = \max\{1-|x|,0\}$.} 
  \end{figure}

  \begin{figure}[h]
      \centering
      \includegraphics[width=0.5\textwidth]{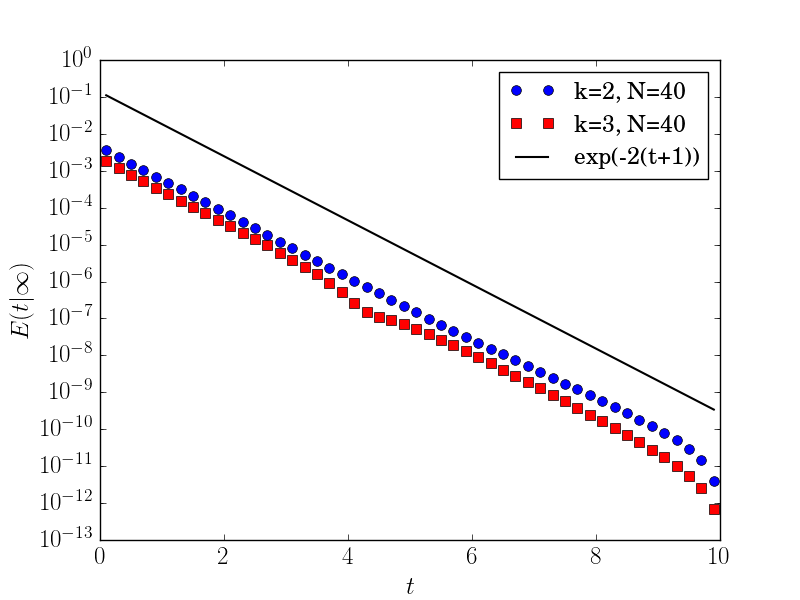}
      \caption{The relative entropy for porous equation, with $\rho(x,0) = 
      \max\{1-|x-\frac{1}{2}|,0\}$.}\label{fig:nonsym_porous_entropy}
  \end{figure}

  For the symmetric initial condition, our numerical tests indicate the decay rates 
  is around $e^{-6.7t}$. Indeed, symmetric initial data converge faster to equilibrium than the sharp rate since they preserve the invariance of the center mass, see \cite{CDT} for more details.
  We then test the problem under the same settings, except for the initial condition 
  shifted to the right $\rho(x,0) = \max\{1-|x-\frac{1}{2}|,0\}$ and the final time set 
  to $t=10$. The corresponding plot of $E(t|\infty)$ is given in Figure 
  \ref{fig:nonsym_porous_entropy} with the exponential decay rate $-2$, which
  coincides with the result in \cite{carrillo2000asymptotic}. Similar numerical
  test can be found in \cite{bessemoulin2012finite}.  
  
\end{examp}
%----------------------------------------------------------------------------------------
\begin{examp}[Fokker-Planck equation] 
  In this numerical test, we consider the Fokker-Planck equation for modeling the 
  relaxation of fermion and boson gases. The equation takes the form 
  \[\partial_t \rho = \partial_x\left(x\rho\left(1+\kappa\rho\right)+\partial_x
  \rho\right).\] 
  Here $\kappa = 1$ corresponds to boson gases and  $\kappa = -1$ relates to fermion
  gases. The long time asymptotics of the one dimensional model has been studied
  in \cite{carrillo2008nonlinear} for $0\leq\rho\leq 1$. The authors point out
  the equation evolves to a steady state $\rho_\infty(x) = \frac{1}{\beta
  e^{\frac{x^2}{2}}-\kappa}$. The stationary solution minimizes the entropy
  functional 
  \[E = \int\left(\frac{|x|^2}{2}\rho + \rho \log(\rho) - \kappa
  (1+\kappa\rho)\log(1+\kappa\rho)\right) dx.\] 
  The relative entropy decays at an exponential rate $E(t|\infty) \leq
  E(0|\infty) e^{-2Ct}$, with $C = 1$ for the boson case and $0<C<1$ for 
the fermion
  case.  In our numerical test, we study the same entropy functional and set
  $f(\rho) = \rho(1+\kappa\rho)$, $H'(\rho) =
  \log\left(\frac{\rho}{1+\kappa\rho}\right)$, $V = \frac{x^2}{2}$ in our
  numerical scheme. The limiter is turned on in the computation. The initial
  condition is chosen as $\rho(x,0) = \frac{1}{0.4\pi}e^{-\frac{(x-1)^2}{0.4}}$.
  We compute on the domain $[-10,10]$ with $100$ mesh cells, and march towards
  the steady state with $\tau = 0.0002h^2$. The flux constant is set to $2$. 
  
  For both of the test cases, we use numerical steady states as references to 
  calculate the relative entropy. The
  result for the boson case is given in Figure \ref{fig:boson_entropy}. As one can
  see, the decay rate is around $-2.6$. While for the fermion case, which is
  exhibited in Figure \ref{fig:fermi_entropy}, the relative entropy decays at a
  slower rate of $-1.44$. 
  \begin{figure}[h!] 
    \centering 
    \begin{subfigure}[h]{0.45\textwidth}
      \centering
      \includegraphics[width=\textwidth]{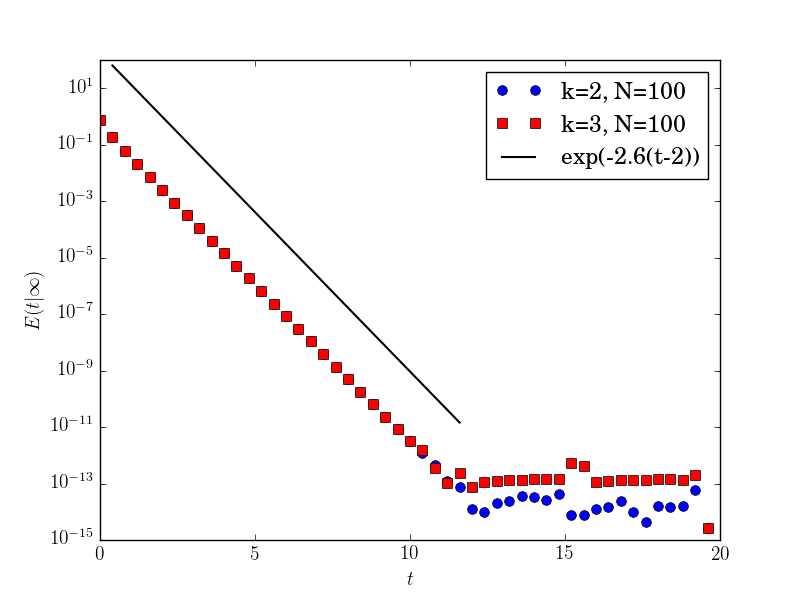}
      \caption{Bose gas.}\label{fig:boson_entropy} 
    \end{subfigure}
    ~
    \begin{subfigure}[h]{0.45\textwidth}
      \centering
    \includegraphics[width=\textwidth]{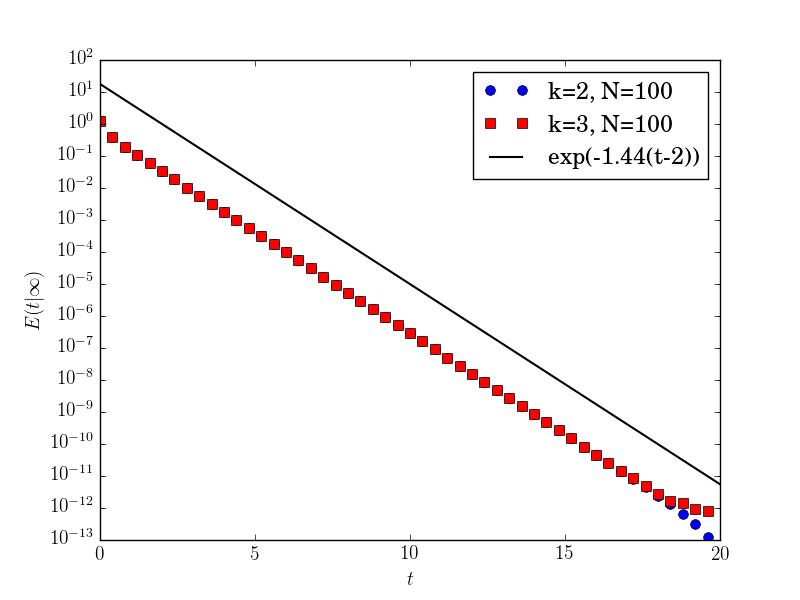}
    \caption{Fermi gas.}\label{fig:fermi_entropy}
    \end{subfigure}
    \caption{Decay of the relative entropy for Fokker-Planck equation.} 
  \end{figure}
\end{examp}
%------------------------------------------------------------------------------------
\begin{examp}[generalized Fokker-Planck equation for the boson gas] 
  Let us now consider the generalized Fokker-Planck equation with linear diffusion and
  superlinear drift 
  \begin{equation*} 
    \partial_t\rho = \partial_x\left(x\rho(1+\rho^N)+\partial_x \rho\right), 
  \end{equation*} 
  with $N$ being a positive constant. For $N>2$, it is reported in
  \cite{abdallah2011minimization} that a critical mass phenomenon exists for
  one dimensional problems.  An initial distribution with supercritical
  mass will evolve a singularity at the origin, which has been confirmed
  numerically in \cite{bessemoulin2012finite} and \cite{liu2016entropy}.  In
  this test, we repeat the numerical experiment in
  \cite{bessemoulin2012finite} and \cite{liu2016entropy}, setting 
  \[f(\rho) = \rho\left(1+\rho^3\right), H'(\rho) = \log\frac{\rho}{\sqrt[3]{1+\rho^3}}
  \text{ and } V(x) = \frac{x^2}{2}.\] 
  The initial datum is chosen as 
  \[\rho(x,0) = \frac{M}{2\sqrt{2\pi}}\left(e^{-\frac{(x-2)^2}{2}} + e^{-\frac{(x+2)^2}{2}}
  \right).\] 
  We test with both the subcritical case with $M= 1$ and the supercritical case with 
  $M = 10$. The $P^4$ elements are used in our numerical scheme and we compute on 
  the domain $[-6,6]$ with $N = 120$. For $M = 1$, the time step is chosen as 
  $\tau = 0.003h^2$ and for $M = 10$, it is $\tau = 0.0005h^2$. And we use 
  $g(\rho) = f(\rho)$ when defining the numerical flux. According to the 
  numerical results in Figure \ref{fig:GFP}, our scheme does capture the 
  asymptotics of the equation.
  \begin{figure}[h!] 
    \centering 
    \begin{subfigure}[h]{0.45\textwidth}
      \centering
    \includegraphics[width=\textwidth]{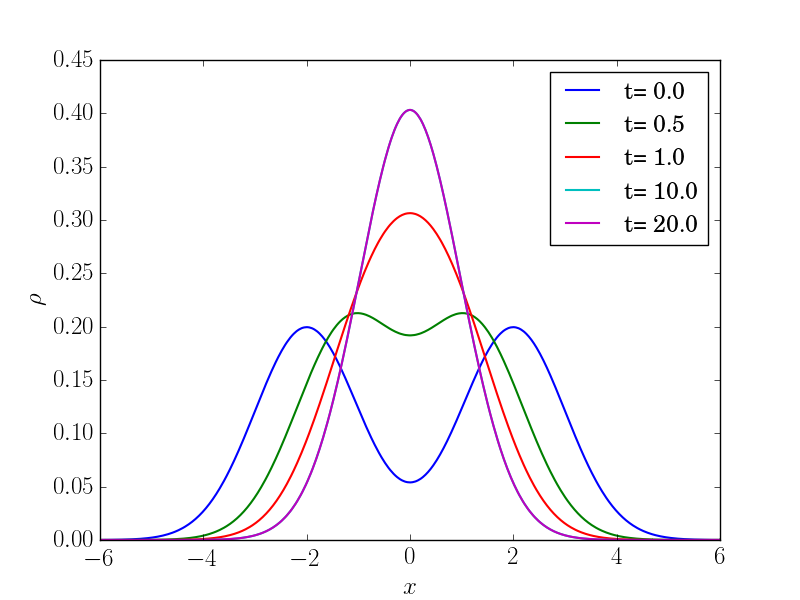}
    \caption{{$M=1$.}\label{fig:GFP_subcritical} }
    \end{subfigure}
     ~
    \begin{subfigure}[h]{0.45\textwidth}
      \centering
      \includegraphics[width=\textwidth]{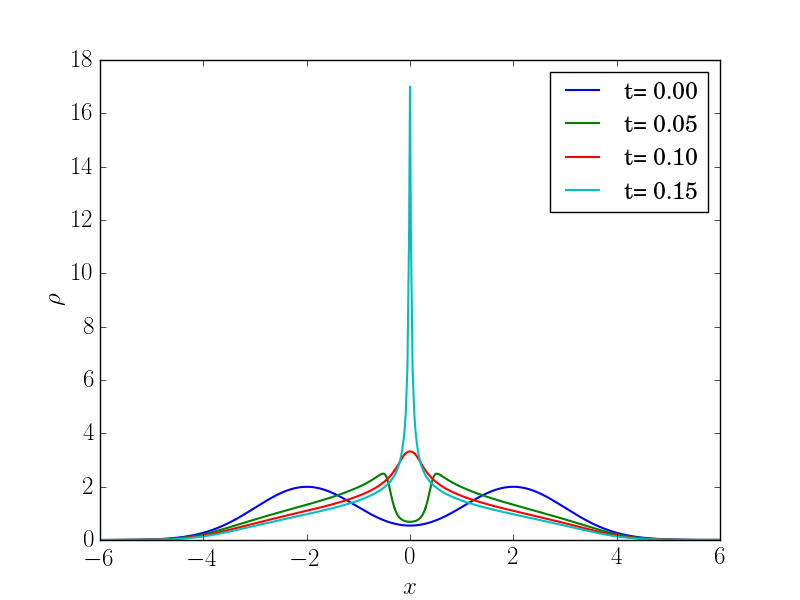}
      \caption{{$M=10$.}\label{fig:GFP_supercritical} }
    \end{subfigure}
    \caption{Evolution of $\rho$ of subcritical mass $M=1$ and supercritical mass
    $M=10$.}\label{fig:GFP} 
  \end{figure}
\end{examp}

\subsection{Aggregation models}
%------------------------------------------------------------------------------------
\begin{examp}[nonlinear diffusion with smooth attraction kernel] 
  This numerical test is to study the dynamics of the equation with competing 
  nonlinear diffusion and smooth nonlocal attraction, 
  \[\partial_t\rho = \partial_x\left(\rho\partial_x\left(\nu\rho^{m-1} + 
  W\ast\rho\right)\right).\]
  Here $0\leq\rho\leq1$. $\nu>0$ and $m>1$ are parameters to be specified. The
  convolution kernel $W$ in this example is nonlocal and smooth. Under this
  setting, the attraction effect is weak and the solution would either end up
  with a steady state or spread out in the whole domain with bounded initial
  data. The compactly supported steady state is of special interest, due to its
  application for modeling the biological aggregation, such as flocks and
  swarms. Indeed, such stationary solution can be reached for $m>2$ with
  arbitrary $\nu$. While for $1<m\leq2$, the long time behavior of the solution
  can be sophisticated. We refer to \cite{burger2014stationary} and
  \cite{burger2013stationary} for details.  For our numerical test, we focus on
  the specific setting, 
  \begin{equation*} 
    \left\{ 
      \begin{aligned} 
        \partial_t\rho &= \partial_x(\rho\partial_x(\nu\rho^{m-1} + W\ast\rho)), 
        W(x) = -\frac{1}{\sqrt{2\pi}}e^{-\frac{x^2}{2}},x\in [-6,6],\\ 
        \rho(x,0) &= \frac{1}{2\sqrt{2\pi}}\left(e^{-\frac{(x-\frac{5}{2})^2}{2}}
        +e^{-\frac{(x+\frac{5}{2})^2}{2}}\right).
      \end{aligned}
    \right.  
  \end{equation*} 
  We apply periodic boundary conditions and use a mesh with $N = 120$ for 
  computation. The $P^2$ scheme is used for the numerical test, and
the time step is 
  chosen as $\tau = 0.05h^2$. In Figure \ref{fig:nonlocal_smooth_steady}, we depict 
  the numerical solution at $T=1800$ with $m=1.5$, $\nu=0.33$, $m=2$, $\nu=0.48$ and 
  $m=3$, $\nu=1.48$, which are used as the reference steady states when evaluating the 
  relative entropy.

  According to the plot, one can see that a larger $m$ corresponds to a steady 
  state with a sharper transition along the boundary of the support. Indeed, one 
  should expect the H\"{o}lder continuity with the exponent $\alpha=\min\{1,1/(m-1)\}$. 

  \begin{figure}[h!] 
    \centering
    \includegraphics[width=0.5\textwidth]{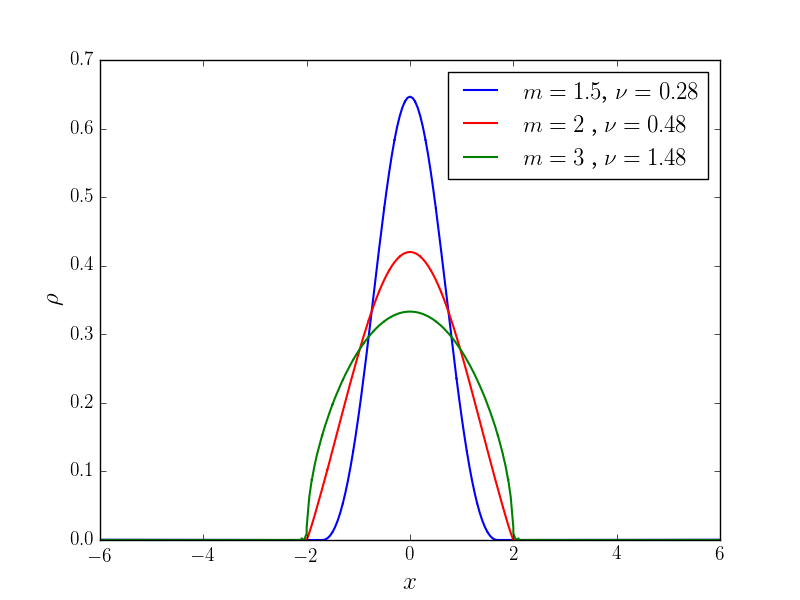}
    \caption{Solution profile at $T=1800$.} \label{fig:nonlocal_smooth_steady}
  \end{figure} 
  \begin{figure}[h!] 
    \centering 
    \begin{subfigure}[h]{0.45\textwidth}
      \centering
    \includegraphics[width=\textwidth]{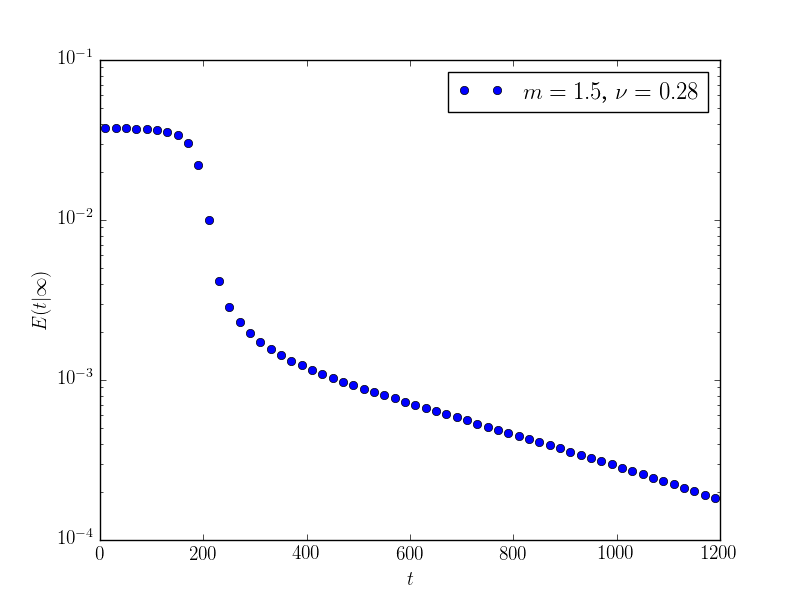}
    \caption{Relative entropy.}\label{fig:log_entropy_nonlocal_smooth_m1d5}
    \end{subfigure}
    ~ 
    \begin{subfigure}[h]{0.45\textwidth}
    \centering
    \includegraphics[width=\textwidth]{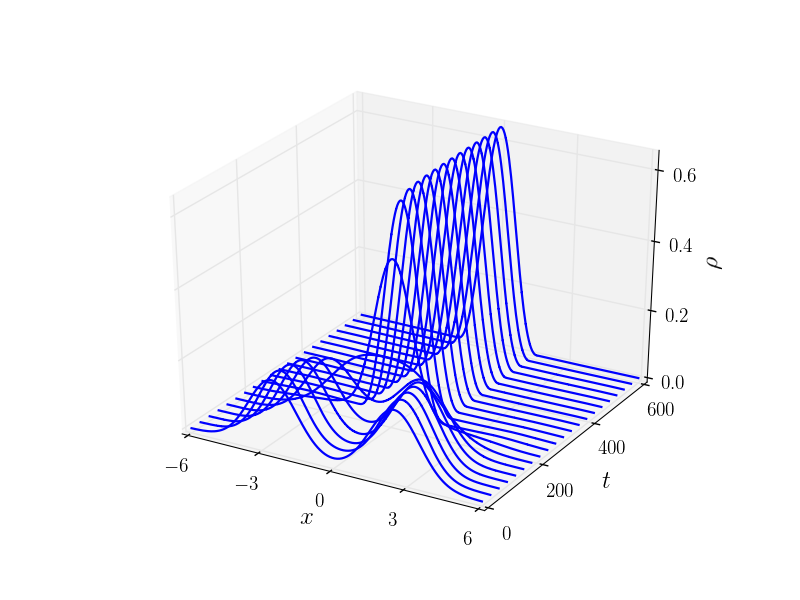}
    \caption{{Evolution.}\label{fig:nonlocal_smooth_m1d5_0d28} }
    \end{subfigure}
    \caption{$m = 1.5$, $\nu = 0.28.$}
  \end{figure} 
  \end{examp} 
  We track the solution profile and the relative entropy. For $m=1.5$, as one can 
  see from Figure \ref{fig:log_entropy_nonlocal_smooth_m1d5}, the dynamics of the 
  problem distinguishes from the test cases for Fokker-Planck type equations. The 
  relative entropy decays slowly at first, then it follows with a steep 
  drop at a certain time. After that, the relative entropy decays exponentially. The
  behavior can be explained with Figure \ref{fig:nonlocal_smooth_m1d5_0d28}. At
  the beginning, the two bumps of the initial condition stay away from each
  other, their interaction is weak hence the equation evolves at a slow rate. When 
  they get closer, the attraction becomes strong.  A sudden decay of the relative 
  entropy occurs when the two bumps merge. After that, the contribution of the 
  interaction potential to the total energy becomes small. The equation is again 
  dominated by the diffusion term, and the relative entropy decays exponentially 
  as we have seen before. 

  We omit the plots for $m=2$. And for $m=3$, the diffusion is relatively weak and 
  the exponential decay after the steep drop is hard to observe. Hence we only plot 
  the entropy in the normal scale. But still we can see the sharp drop when the bumps 
  merge in Figure \ref{fig:nonlocal_smooth_m3}.

  The initial stage featured with the weak long-range-interaction is referred
  as metastability. If multiple bumps exist, the relative entropy can decay in
  a staircase fashion. For example, we test the problem with $m=6$, $\nu=6$ and
  $\rho(x,0)= \frac{3}{20}\max\{1-|x-\frac{19}{4}|,0\}+
  \frac{1}{4}\max\{1-|x-2|,0\}+ \frac{1}{5}\max\{1-|x+\frac{17}{40}|,0\}+
  \frac{2}{5}\max\{1-|x+\frac{15}{4}|,0\}$.  The relative entropy and dynamics
  are given in Figure \ref{fig:nonlocal_step}.
  
  \begin{figure}[h!] 
    \centering 
    \begin{subfigure}[h]{0.45\textwidth}
    \centering
    \includegraphics[width=\textwidth]{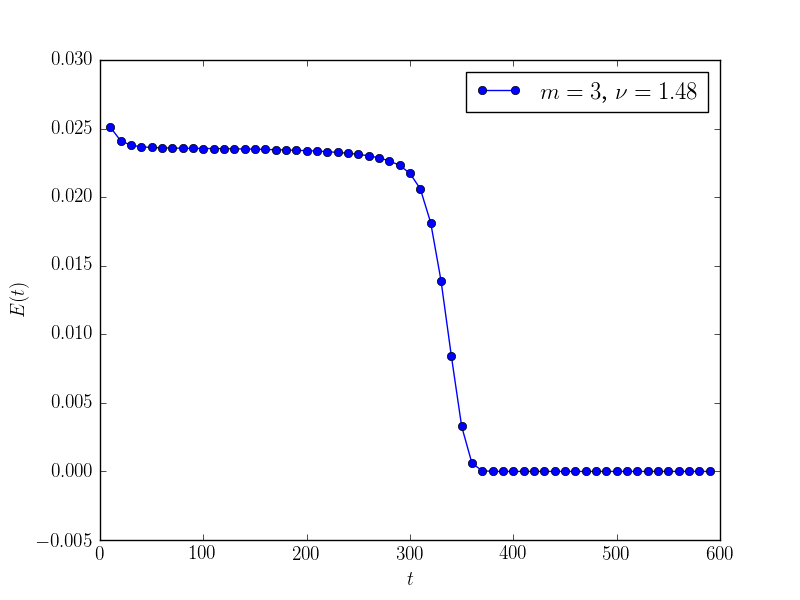}
    \caption{{Entropy.}\label{fig:entropy_nonlocal_smooth_m3}}
    \end{subfigure}
    ~ 
    \begin{subfigure}[h]{0.45\textwidth}  
    \centering
    \includegraphics[width=\textwidth]{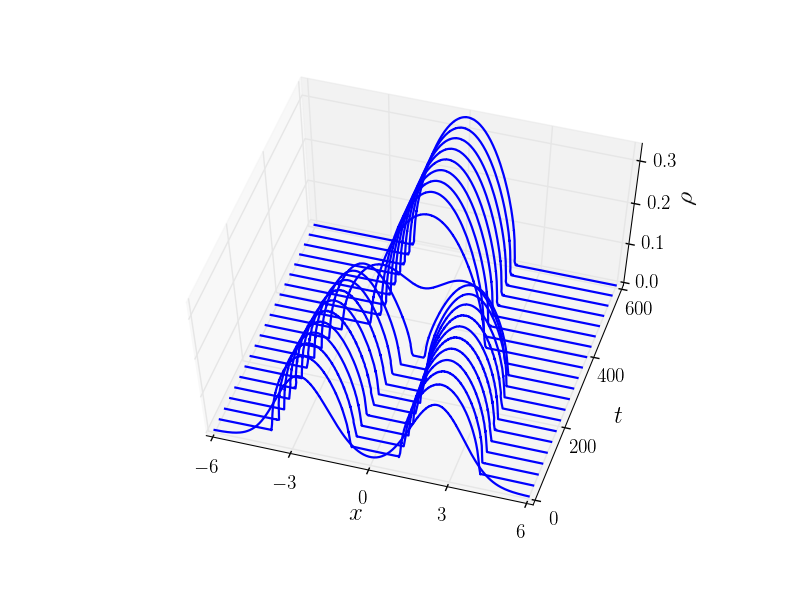}
    \caption{{Evolution.}\label{fig:nonlocal_smooth_m3_1d48} }
    \end{subfigure}
    \caption{$m = 3$, $\nu = 1.48$.}\label{fig:nonlocal_smooth_m3}
  \end{figure} 
  \begin{figure}[h!] 
    \centering
    \begin{subfigure}[h]{0.45\textwidth}
    \centering
      \includegraphics[width=\textwidth]{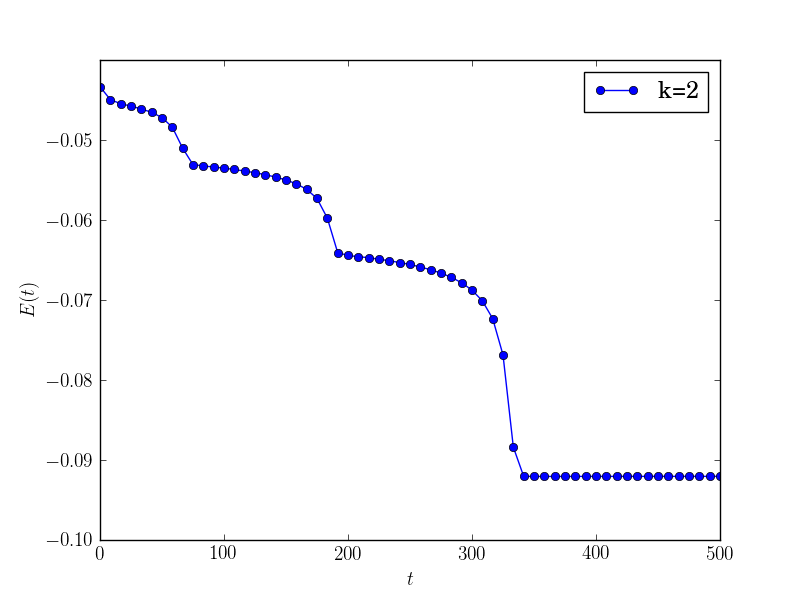}
      \caption{Entropy.}\label{fig:entropy_nonlocal_step} 
    \end{subfigure}
    ~ 
    \begin{subfigure}[h]{0.45\textwidth}
    \centering
    \includegraphics[width=\textwidth]{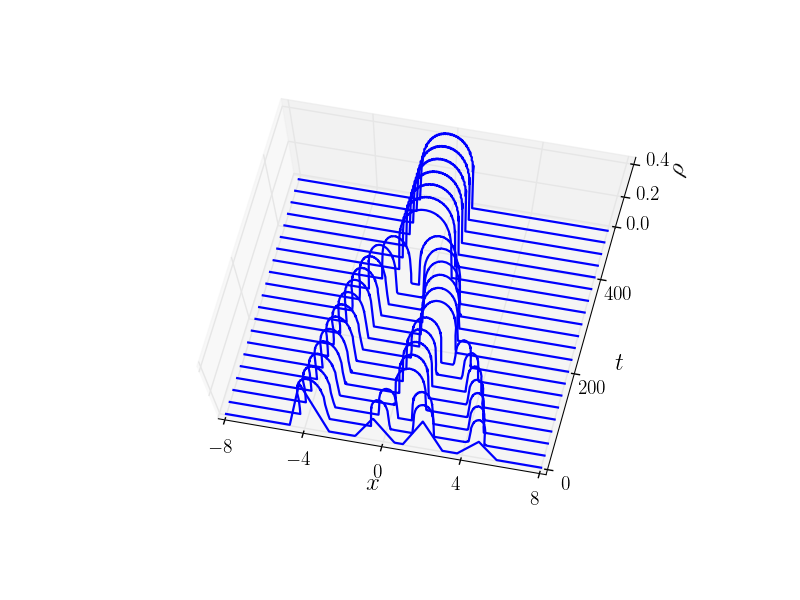}
    \caption{Evolution.} 
    \end{subfigure}
    \caption{Stepwise decay: $m = 6$, $\nu = 6$.}\label{fig:nonlocal_step} 
    \end{figure} 
\end{examp} 
  %------------------------------------------------------------------------------
\begin{examp}[nonlinear diffusion with compactly supported attraction kernel] 
  In the previous test, the attraction effect is global and the steady state will 
  be connected for one dimensional problems. But when $W$ is local, the connectivity 
  of the equilibrium can be affected by the initial mass distribution. Let us consider 
  the following problem
  \begin{equation}\label{eq-nonsmooth-local-kernel} 
    \left\{ 
      \begin{aligned}
        \partial_t\rho &= \partial_x\left(\rho\partial_x\left(\frac{1}{4}\rho^2 +
        W\ast \rho\right)\right),\quad W = -\max\{1-|x|,0\},\quad x\in [-4,4],\\
        \rho(x,0) &=\chi_{[-a,a]}(x), 
      \end{aligned}\right.  
    \end{equation} 
  with periodic boundary conditions. We compute with $k =2$ with the number of cells 
  $N = 80$. 

  \begin{figure}[h!] 
    \centering 
    \begin{subfigure}[h]{0.45\textwidth}
    \centering
    \includegraphics[width=\textwidth]{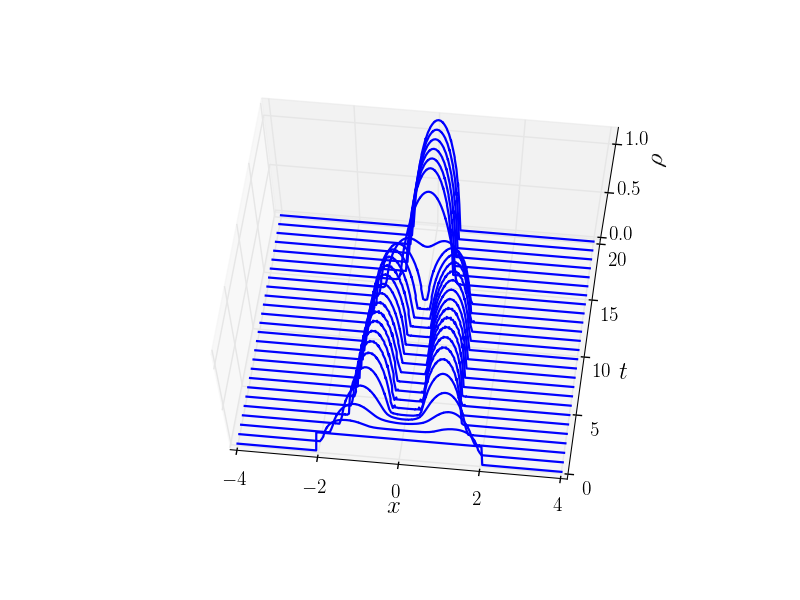}
      \caption{$\rho(x,0) = \chi_{[-2,2]}(x)$}\label{fig:local_connected_history.png}
    \end{subfigure}
     ~ 
     \begin{subfigure}[h]{0.45\textwidth}
    \centering
     \includegraphics[width=\textwidth]{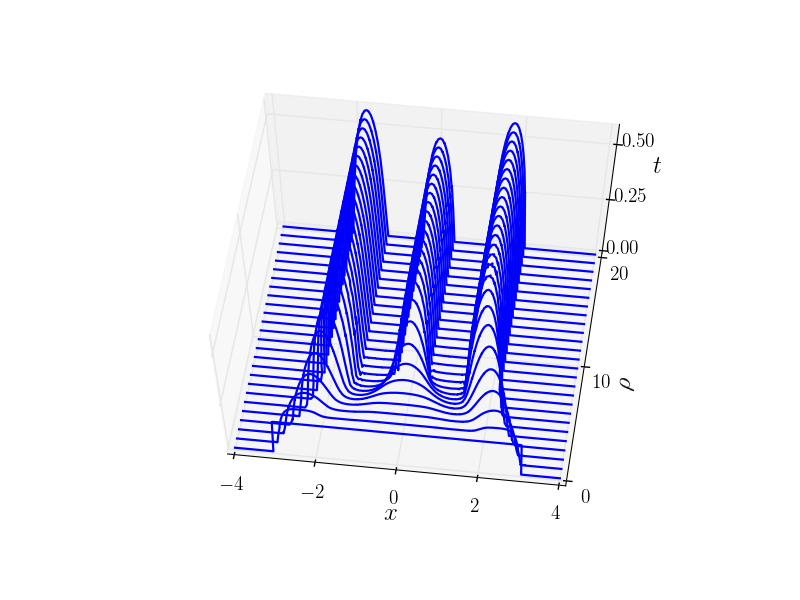}
      \caption{$\rho(x,0) =\chi_{[-3,3]}(x)$}\label{fig:local_disconnected_history.png}
    \end{subfigure}
    \caption{Evolution of \eqref{eq-nonsmooth-local-kernel} with different initial
    conditions.} \label{fig:connectivity}
  \end{figure}
  \begin{figure}[h!] 
    \centering 
    \begin{subfigure}[h]{0.45\textwidth}
    \centering
    \includegraphics[width=\textwidth]{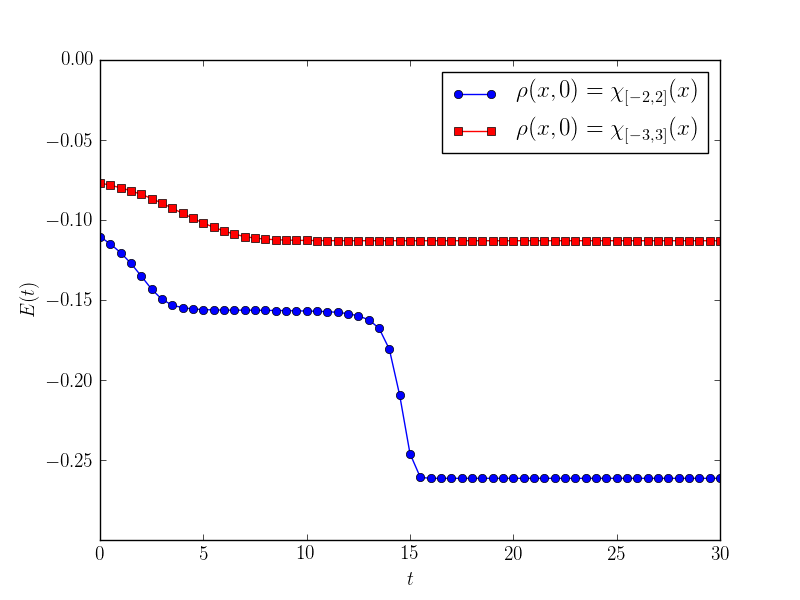}
      \caption{Entropy}\label{fig:entropy_local.png} 
    \end{subfigure}
    ~
    \begin{subfigure}[h]{0.45\textwidth}
    \centering
    \includegraphics[width=\textwidth]{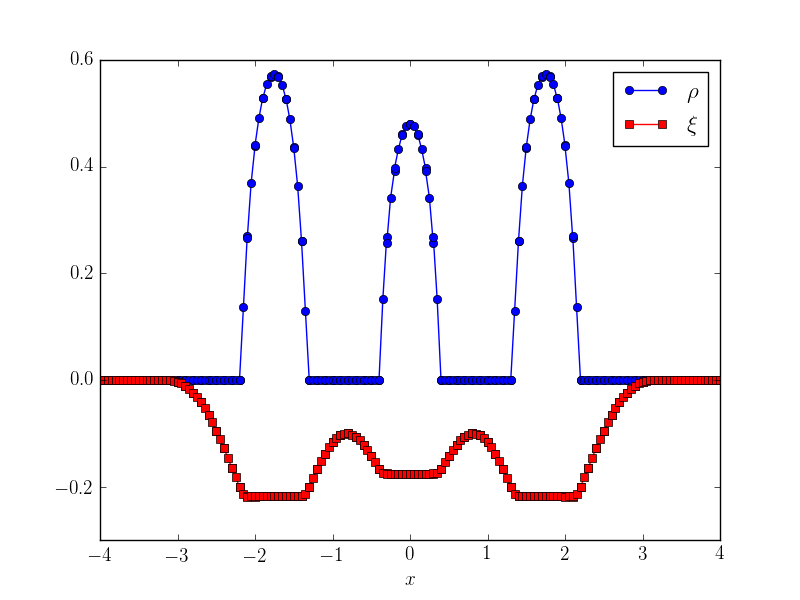} 
      \caption{$\xi$ and $\rho$ at $t = 30$ for $\rho(x,0) = \chi_{[-3,3]}(x)$
      \label{fig:disconnected_rho_xi.png} }
    \end{subfigure}
    \caption{Entropy and the steady state of \eqref{eq-nonsmooth-local-kernel}.} 
  \end{figure}
  To convince the readers that the disconnected profile in Figure 
  \ref{fig:local_disconnected_history.png} is indeed the stationary solution, we plot 
  $\rho$ and $\xi$ in Figure \ref{fig:disconnected_rho_xi.png}. As one can observe,
  $\rho\partial_x \xi \approx 0$. Hence $\partial_t\rho \approx 0$ and $\rho$ will be 
  trapped in this steady state. Therefore, the observation in Figure 
  \ref{fig:connectivity} confirms our previous claim, that different initial density 
  distributions may end up with steady states with distinct connectivity. Let us 
  remark that such phenomenon has also been explored numerically in 
  \cite{carrillo2015finite}.
\end{examp}
%---------------------------------------------------------------------------------
%
\section{Two dimensional numerical tests}
\label{sec-num-result-2d}
\setcounter{equation}{0} 
\setcounter{figure}{0} 
\setcounter{table}{0}
%In this section, we present selected numerical tests for our scheme for two
%dimensional problem.
\begin{Examp}[accuracy test] 
  We consider the initial value problem with a source term, 
  \begin{equation}\label{eq-2d-accuracy} 
    \left\{ 
      \begin{aligned}
        \partial_t \rho &= \nabla\cdot\left(\rho\nabla\left(\log(\rho) +\sin(x+y) 
        + W\ast\rho\right)\right) + F, (x,y)\in (-\pi,\pi)\times(-\pi,\pi),t>0\\ 
        W(x,y) &= \cos(x+y),\\ 
        F(x,y) &= 4\sin(x+y)+\cos(x+y+t)+\left(2+8\pi^2\right)\sin(x+y+t)\\
        &-2\cos\left(2(x+y)+t\right)-4\pi^2\cos\left(2\left(x+y+t\right)\right)\\
        \rho(x,y,0) &= \sin(x+y)+2.\\ 
      \end{aligned} 
    \right.
  \end{equation} 
  Here periodic boundary conditions are applied and $W\ast\rho =
  \int_{-\pi}^{\pi}\int_{-\pi}^{\pi}W(x-\xx,y-\yy)\rho(\xx,\yy)d\xx d\yy$. One
  can check that the exact solution to \eqref{eq-2d-accuracy} is $\rho(x,y,t) =
  2+\sin(x+y+t)$. We use the time step $\tau = 0.0005(h^x)^2$ for 
the calculation.  The error table is given in Table \ref{tab:2D_accuracy}.  
  \begin{table}[h!]
    \centering 
    \begin{tabular}{c|c|c|c|c|c|c|c} 
      \hline   
      k &N&$L^1$ error& order&$L^2$ error& order&$L^\infty$ error& order\\ 
      \hline 
      1&$10\times 10$ &5.11679     & -   & 1.02912      & -    & 0.400769     & -   \\ 
       &$20\times 20$& 1.15240     & 2.15& 0.231872     & 2.15 & 0.944210E-01 &2.09  \\
       &$40\times 40$& 0.301086    & 1.94& 0.621649E-01 & 1.90 & 0.241975E-01 &1.96  \\  
      \hline 
      2&$10\times 10$& 1.45276     & -   & 0.306948     & -& 0.167537    &  -   \\ 
       &$20\times 20$& 0.222326    & 2.71& 0.431470E-01 & 2.83 & 0.297235E-01 &2.49  \\  
       &$40\times 40$& 0.356271E-01& 2.64& 0.720268E-02 & 2.58 & 0.513506E-02 &2.53\\ 
      \hline 
      3&$10\times 10$& 0.377751E-01& -   & 0.792888E-02 & -    & 0.439644E-02 &-  \\  
       &$20\times 20$& 0.229221E-02& 4.04& 0.525495E-03 & 3.92 & 0.294146E-03 &3.90  \\
       &$40\times 40$& 0.137322E-03& 4.06& 0.333325E-04 & 3.98 & 0.205418E-04 &3.83  \\ 
      \hline 
      4&$10\times 10$& 0.224001E-02& -   & 0.511292E-03 & - & 0.294120E-03 &-   \\ 
       &$20\times 20$& 0.676477E-04& 5.05& 0.143874E-04 & 5.15 & 0.115235E-04 &4.67   \\ 
       &$40\times 40$& 0.243927E-05& 4.80& 0.524241E-06 & 4.78 & 0.450331E-06 &4.68   \\ 
     \hline 
   \end{tabular}
  \caption{Two dimensional accuracy test, with smooth data and periodic boundary
  conditions.}\label{tab:2D_accuracy} 
  \end{table} 
  \end{Examp}
  %---------------------------------------------------------------------------------------
\begin{Examp}[dumbbell model for polymers] 
  The dumbbell model is widely used to describe the rheological behavior of
  dilute polymer solutions. In this model, the polymer molecular is treated as
  a dumbbell made of two beads jointed by a spring. We will consider the
  simplest case, in which the flow is homogeneous and the scaling constant is
  set to $1$. Then the configuration probability density is governed by the
  Fokker-Plank equation,
  \begin{equation}\label{eqn:2Ddumbbell} 
    \partial_t\rho(\mathbf{x},t) = \nabla
  \cdot (\rho\nabla ( U - \frac{1}{2}\mathbf{x}K\mathbf{x})) + \Delta \rho.
  \end{equation} 
  Here $\mathbf{x} = (x,y)$ corresponds to the direction vector of the
  molecule, while $U$ is the spring potential and the $2\times 2$ matrix $K$
  is the velocity gradient of the background flow. For the incompressible flow,
  $\text{Tr}(K) = 0$.  In our numerical test, we consider the finitely
  extensible nonlinear elastic (FENE) model. The potential $U$ is given by
  \begin{equation}\label{eqn:2DdumbbellFENE} 
    U(\mathbf{x}) = -\frac{r^2}{2}\log\left(1
-\frac{|\mathbf{x}|^2}{r^2}\right). 
  \end{equation} 
  It is close to the Hookean potential when $|\mathbf{x}|\ll r$, while the
  distance between the two beads are restricted within $r$. Rigorously, one
  should consider the equation on the ball $\{\mathbf{x}: |\mathbf{x}|\leq
  r\}$, and the singularity near the boundary will cause challenges both
  analytically and numerically, see \cite{du2005fene,liu2008boundary,
  shen2012approximation,liu2014maximum} and the references therein.  While in
  our numerical test, we only consider a simpler case, that the solution seems
  to be supported within the ball and it hardly reaches the boundaries.  More
  specifically, we solve \eqref{eqn:2Ddumbbell}-\eqref{eqn:2DdumbbellFENE} with
  $r = 5$ and $K = \left(\begin{matrix} 0.3&0.2\\0.2&-0.3\end{matrix}\right)$.
  The initial condition is set as 
  \begin{equation}\label{eqn:2Ddumbbellinit} 
    \begin{split}
    \rho(x,y,0) = &c\max\{24-(x^2+y^2),0\}\\
    &\left(e^{-\frac{(x-2)^2+(y-2)^2}{2\sigma^2}}
      +e^{-\frac{(x+2)^2+(y+2)^2}{2\sigma^2}}
      +e^{-\frac{(x-1)^2+(y-1)^2}{2\sigma^2}}
      +e^{-\frac{(x+1)^2+(y+1)^2}{2\sigma^2}}
    \right),
    \end{split}
  \end{equation} 
  where $\sigma^2 = 0.2$ and $c$ the normalization constant. 
  We use the $P^3$ DG scheme on a $[-5,5]\times[-5,5]$ domain with $50\times
  50$ mesh cells.  The time step is chosen as $\tau = 2\times 10^{-6}$.  In
  Figure \ref{fig:2d-dumbbell}, we plot the evolution from $T=0$ to $T = 0.6$.
  It seems the numerical solution merges to a single peak.  
\end{Examp}
  \begin{figure}[h!] 
    \centering 
    \begin{subfigure}[h]{0.45\textwidth}
    \centering
    \includegraphics[width=\textwidth]{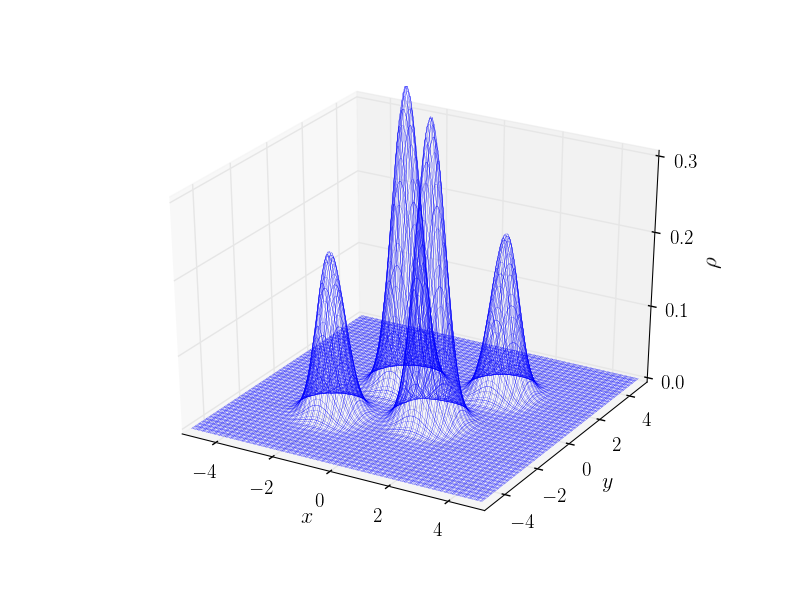} 
    \caption{{$t =0$.}\label{fig:Dumbbell_T0} }
    \end{subfigure}
    ~
    \begin{subfigure}[h]{0.45\textwidth}
    \centering
    \includegraphics[width=\textwidth]{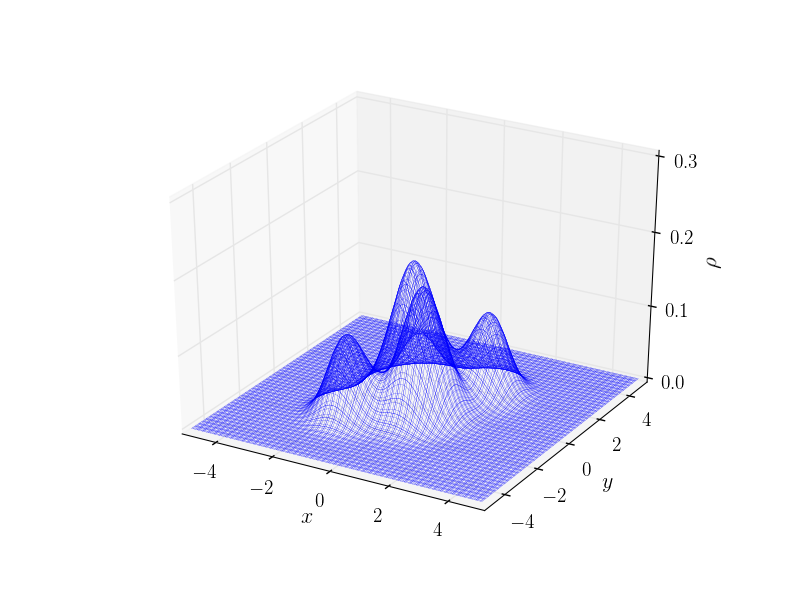}
      \caption{$t=0.2$.}\label{fig:Dumbbell_T0d2} 
    \end{subfigure}
    ~
    \begin{subfigure}[h]{0.45\textwidth}
    \centering
    \includegraphics[width=\textwidth]{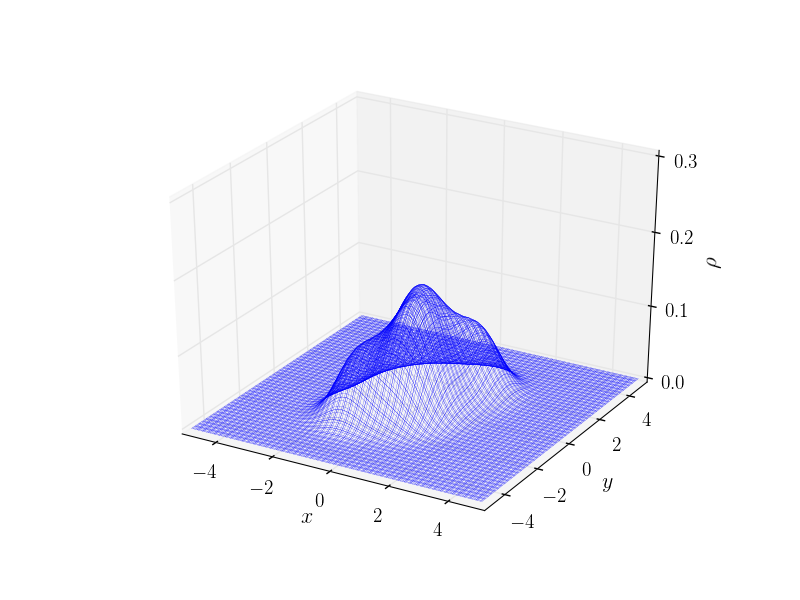}
      \caption{$t=0.4$.}\label{fig:Dumbbell_T0d4} 
    \end{subfigure}
    ~
    \begin{subfigure}[h]{0.45\textwidth}
    \centering
    \includegraphics[width=\textwidth]{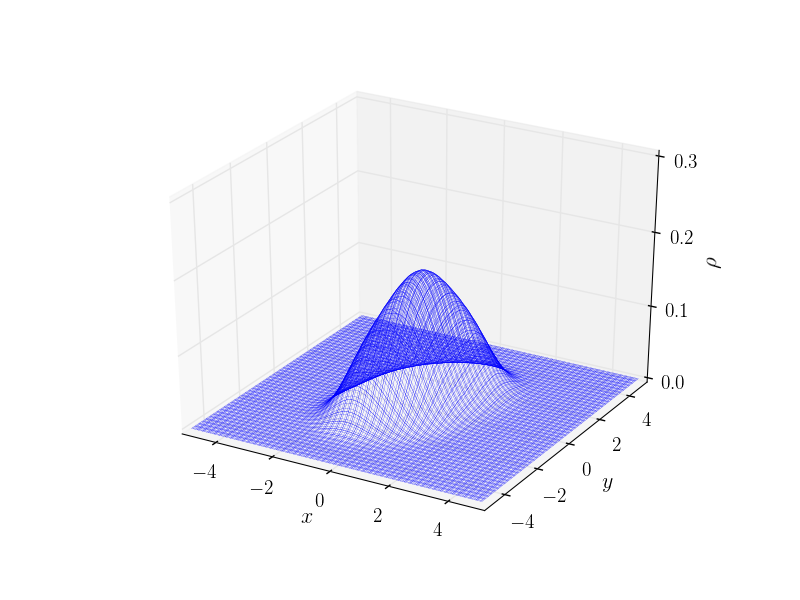}
      \caption{$t=0.6$.}\label{fig:Dumbbell_T0d6} 
    \end{subfigure}
  \caption{Evolution of the dumbbell model \eqref{eqn:2Ddumbbell} with 
  the FENE potential \eqref{eqn:2DdumbbellFENE}.}
  \label{fig:2d-dumbbell} 
\end{figure}
%------------------------------------------------------------------------------
\begin{Examp}[Patlak-Keller-Segel system for chemotaxis] 
  Chemotaxis is defined as a move of an organism along a chemical concentration 
  gradient. Bacteria can produce this chemo-attractant themselves, creating thus a 
  long range nonlocal interaction between them. The Patlak-Keller-Segel system is 
  a mathematical model to describe the motion of the organism. Its simplified version 
  is given by 
  \begin{equation*}%\label{eq-keller-segel} 
    \left\{ 
      \begin{aligned}
        \partial_t\rho &= \Delta \rho - \nabla\cdot (\rho \nabla c),\qquad
        (x,y)\in \mathbb{R}^2, t>0, \\ 
        -\Delta c &=\rho , \qquad (x,y)\in \mathbb{R}^2, t>0, \\ 
        \rho(x,y,0) &= \rho_0(x,y). 
    \end{aligned}\right.
  \end{equation*} 
  The equation can be rewritten in a compact way
  \begin{equation}\label{eq-2d-pks} 
    \partial_t\rho = \nabla\cdot \left(\rho
        \nabla \left(\log(\rho)+W\ast\rho\right)\right), \quad W(x,y) =
        \frac{1}{2\pi}\log(\sqrt{x^2+y^2})\quad (x,y)\in \mathbb{R}^2, t>0.
  \end{equation} 
  Such system has been studied intensively in the past decades.
  It has been shown that the behavior of the equation \eqref{eq-2d-pks} is
  determined by its initial mass (see \cite{blanchet2006twodimensional}, for
  example). If the initial value $M$ is smaller than a critical value $M_c =
  8\pi$, then the solution will exist globally. Otherwise, if $M$ lies beyond
  $M_c$, the solution will blow up in a finite time, which is referred as
  chemotactic collapse.

  \begin{figure}[h!] 
    \centering 
    \begin{subfigure}[h]{0.45\textwidth}
    \centering
    \includegraphics[width=\textwidth]{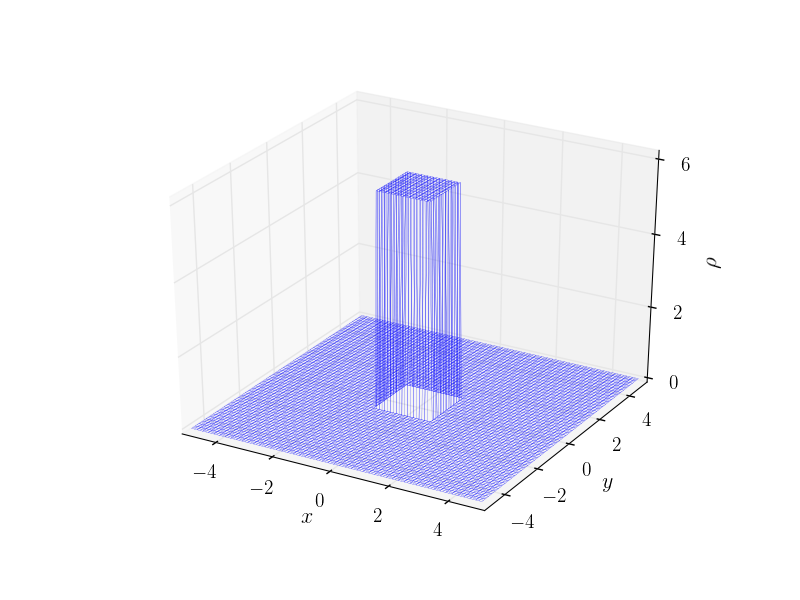} 
      \caption{$t =0$.}\label{fig:KellerSegel_dissip_T0} 
    \end{subfigure}
    ~
    \begin{subfigure}[h]{0.45\textwidth}
    \centering
    \includegraphics[width=\textwidth]{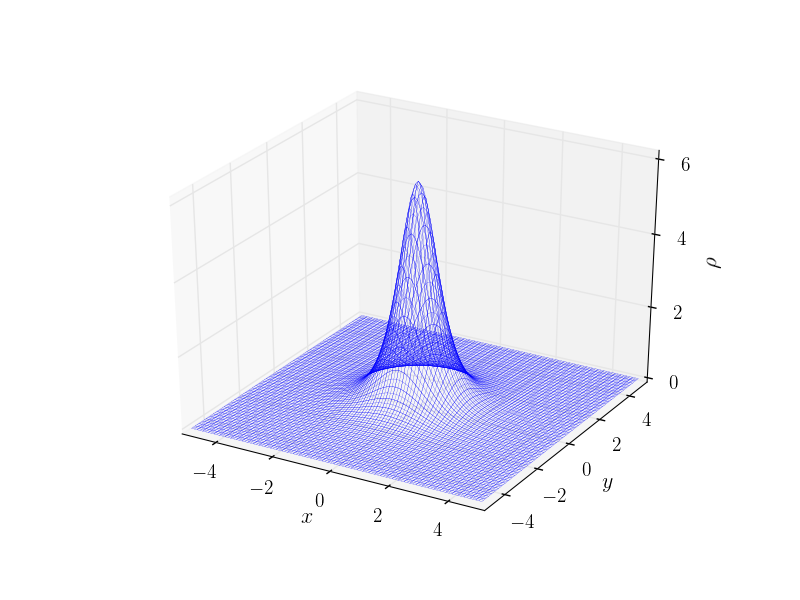}
      \caption{$t=4$.}\label{fig:KellerSegel_dissip_T4} 
    \end{subfigure}
  \vskip\baselineskip 
  \begin{subfigure}[h]{0.45\textwidth}
    \centering
    \includegraphics[width=\textwidth]{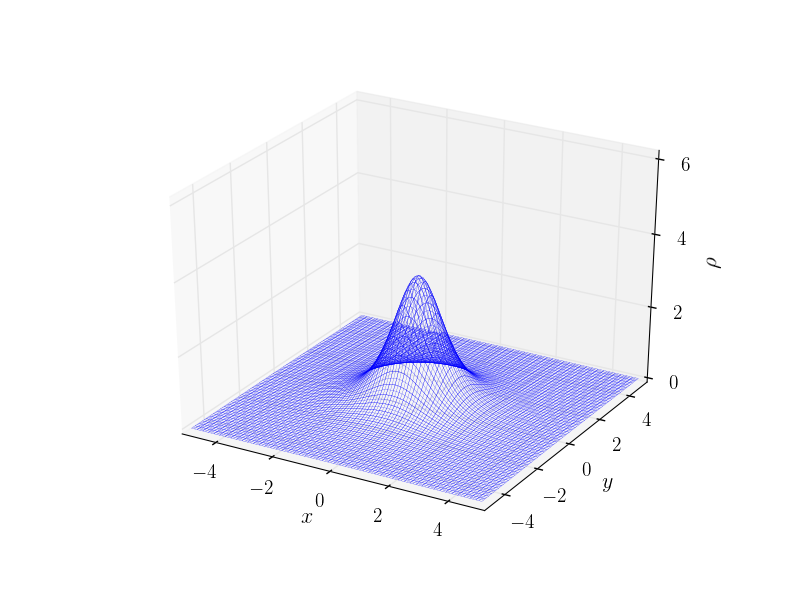}
    \caption{$t=8$.}\label{fig:KellerSegel_dissip_T8} 
  \end{subfigure}
    ~
    \begin{subfigure}[h]{0.45\textwidth}
      \centering
    \includegraphics[width=\textwidth]{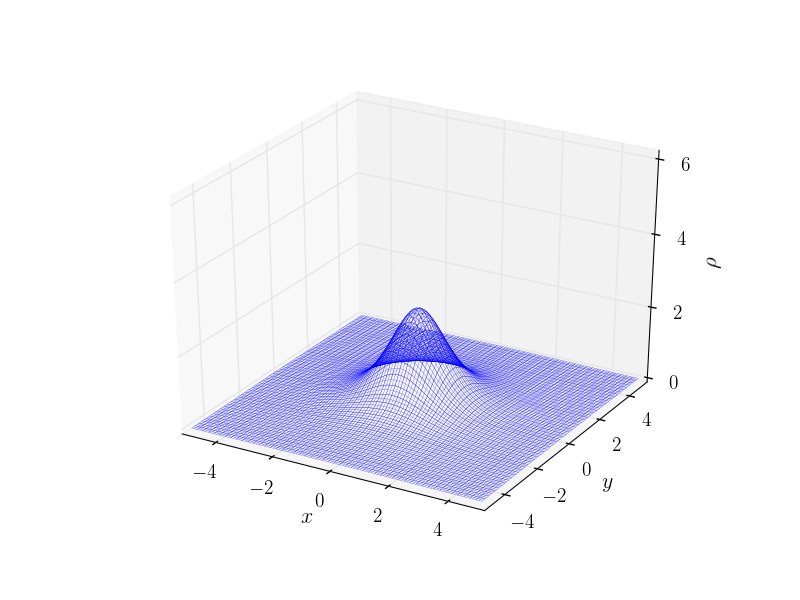}
      \caption{$t=12$.}\label{fig:KellerSegel_dissip_T12}
    \end{subfigure}
  \vskip\baselineskip 
  \begin{subfigure}[h]{0.45\textwidth}
    \centering
    \includegraphics[width=\textwidth]{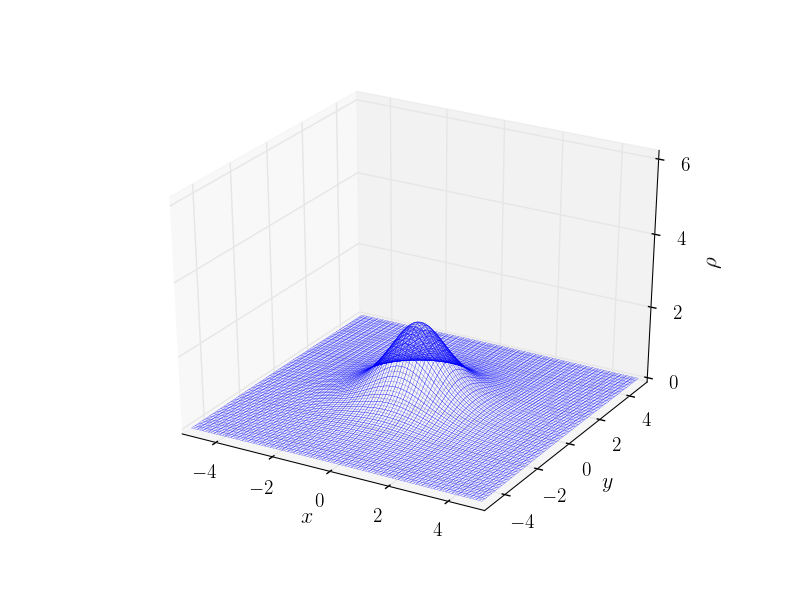}
    \caption{$t=16$.}\label{fig:KellerSegel_dissip_T16} 
  \end{subfigure}
    ~
    \begin{subfigure}[h]{0.45\textwidth}
    \centering
    \includegraphics[width=\textwidth]{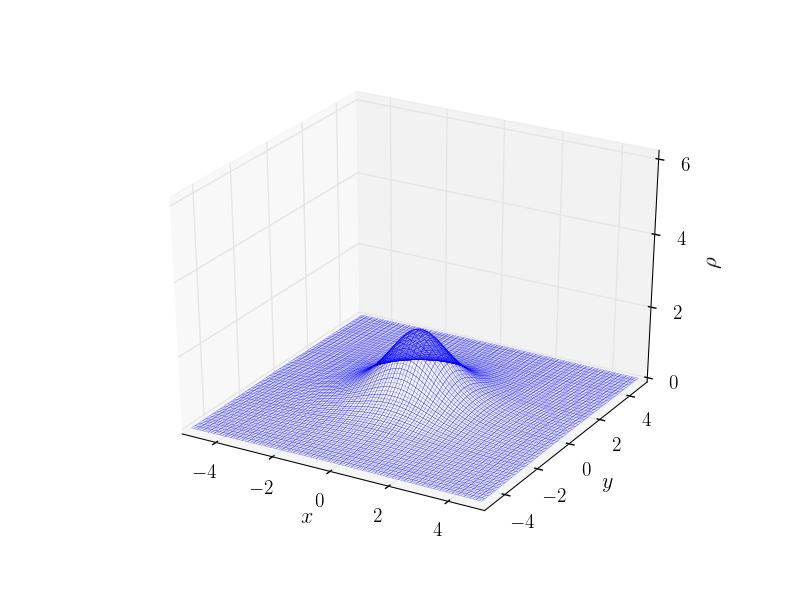}
    \caption{$t=20$.}\label{fig:KellerSegel_dissip_T20}
  \end{subfigure}
  \caption{Evolution of Patlak-Keller-Segel equation \eqref{eq-2d-pks} with
  subcritical mass.}\label{fig:2d-pks-subcritical} 
\end{figure} 
\begin{figure}[h!]
  \centering 
  \begin{subfigure}[h]{0.45\textwidth}
    \centering
    \includegraphics[width=\textwidth]{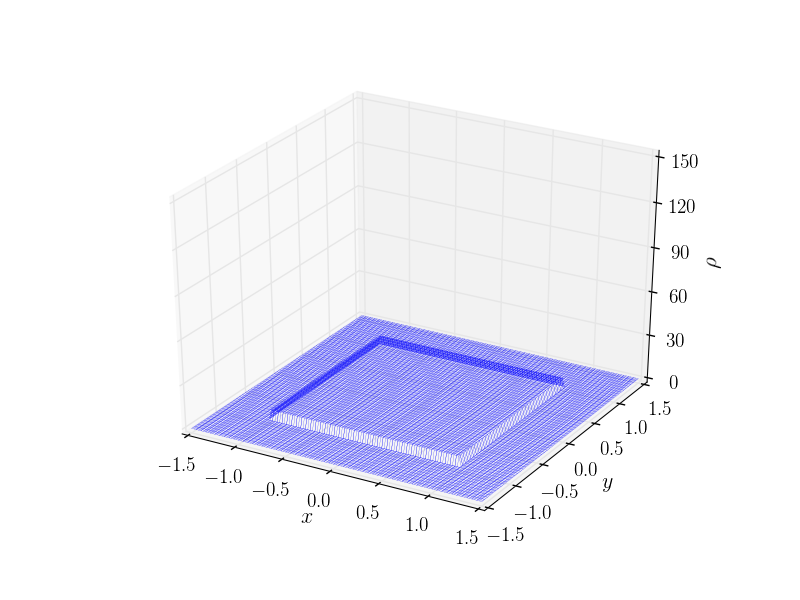}
    \caption{$t =0$.}\label{fig:KellerSegel_T0} 
  \end{subfigure}~
  \begin{subfigure}[h]{0.45\textwidth}
    \centering
    \includegraphics[width=\textwidth]{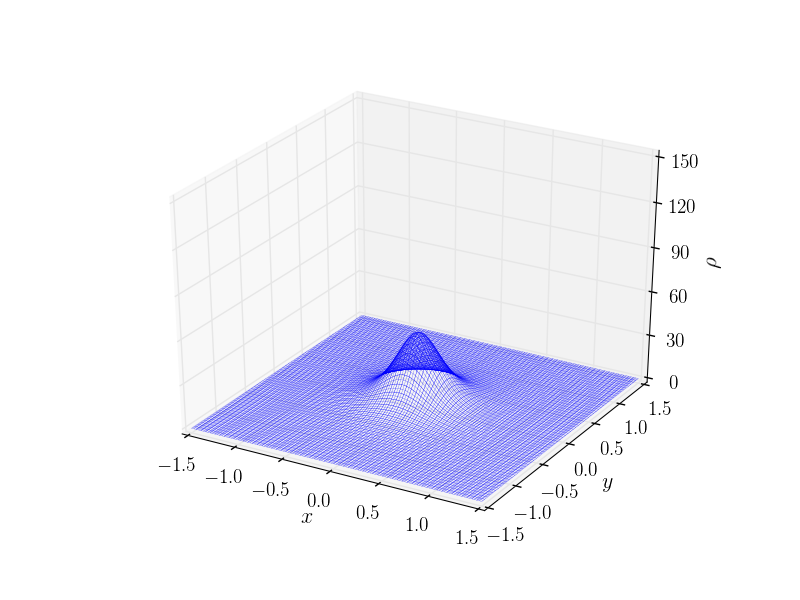}
    \caption{$t=0.5$.}\label{fig:KellerSegel_T0d5} 
  \end{subfigure}
  \vskip\baselineskip 
  \begin{subfigure}[h]{0.45\textwidth}
    \centering
    \includegraphics[width=\textwidth]{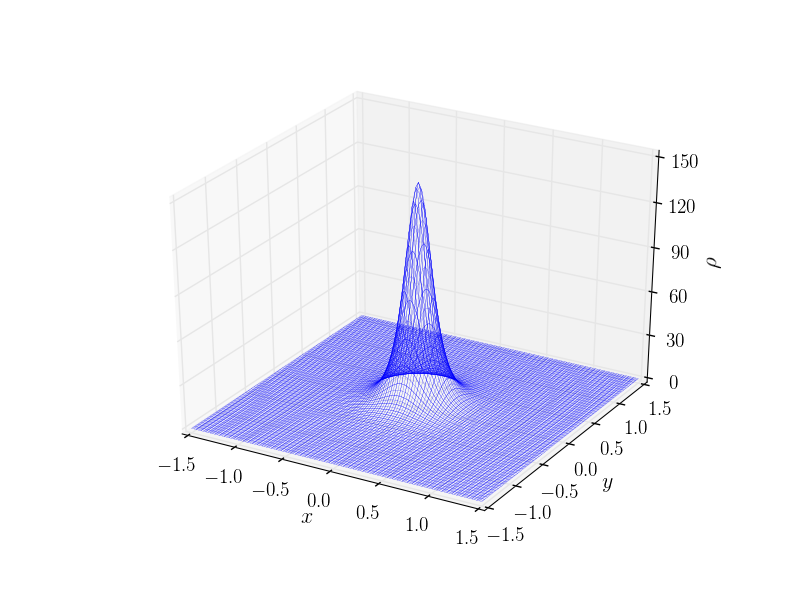}
    \caption{$t=1$.}\label{fig:KellerSegel_T1} 
  \end{subfigure}~
  \begin{subfigure}[h]{0.45\textwidth}
    \centering
    \includegraphics[width=\textwidth]{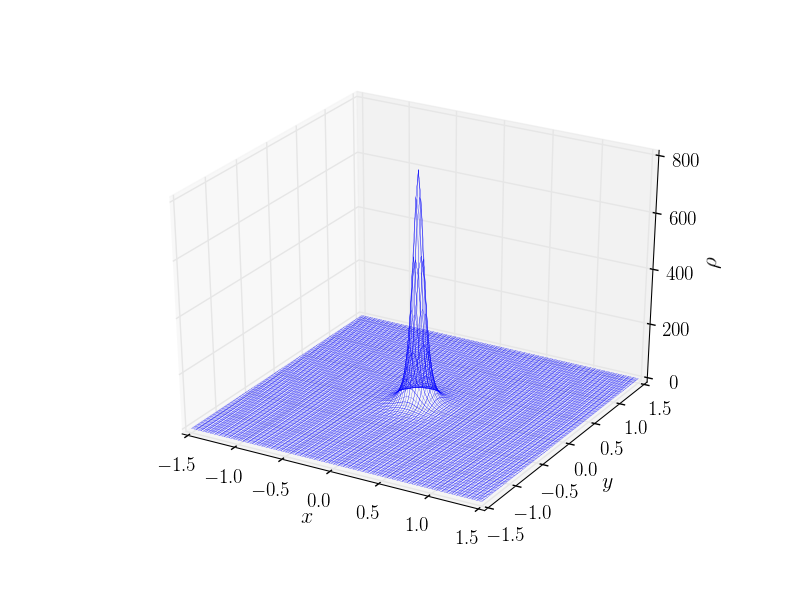}
    \caption{$t=1.5$.}\label{fig:KellerSegel_T1d5}
  \end{subfigure}
  \vskip\baselineskip 
  \begin{subfigure}[h]{0.45\textwidth}
    \centering
    \includegraphics[width=\textwidth]{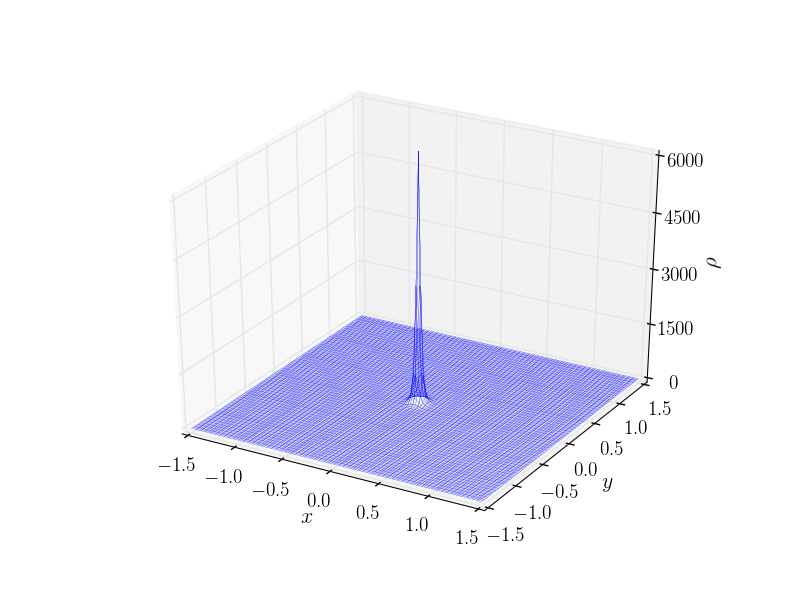}
    \caption{$t=2$.}\label{fig:KellerSegel_T2} 
  \end{subfigure}
  \caption{Evolution of Patlak-Keller-Segel equation 
    \eqref{eq-2d-pks} with supercritical mass.}\label{fig:2d-pks-supercritical} 
\end{figure} 

  In our numerical test, we consider both the subcritical case $\rho_0(x) =
  2(\pi-0.2)1_{[-1,1]\times[-1,1]}(x,y)$ and the super-critical case $\rho_0(x) =
  2(\pi+0.2)1_{[-1,1]\times[-1,1]}(x,y)$. The computational domain is set as
  $[-5,5]\times[-5,5]$ and $[-\frac{3}{2},\frac{3}{2}]\times[-\frac{3}{2},\frac{3}{2}]$ 
  respectively. We use the $P^2$ scheme for computation and $N_x=N_y=50$. 
  The time step is set as $\tau = 0.0005(h^x)^2$. The plots are given in Figure
  \ref{fig:2d-pks-subcritical} and Figure \ref{fig:2d-pks-supercritical}. As one
  can see, the numerical solution dissipates for $\rho_0(x) =
  2(\pi-0.2)1_{[-1,1]\times[-1,1]}(x,y)$ and it evolves to a spike centered at the
  origin for $\rho_0(x) = 2(\pi+0.2)1_{[-1,1]\times[-1,1]}(x,y)$.
\end{Examp}

\section{Concluding remarks} \label{sec-concluding-remarks} 

In this paper, we develop a high order DG method for solving a class of
parabolic equations and gradient flow problems with interaction potentials.
Such equations are governed by an entropy-entropy dissipation relationship and
are featured with non-negative solutions. By applying the Gauss-Lobatto
quadrature rule, our numerical scheme admits an entropy inequality for problems
with smooth interaction kernels. Furthermore, with the SSP-RK time
discretization and the positivity-preserving limiter, the fully discretized
scheme preserves the non-negativity of the numerical density. It also conserves
mass, and preserves numerical steady states for certain problems. We also apply the
method to two dimensional problems on Cartesian meshes. Numerical examples are
given to demonstrate the performance of the scheme.  

%%%%%%%%%%%%%%%%%%%%%%%%%%

 \end{document}